\documentclass[11pt]{article}
\usepackage[margin=1in]{geometry}

\usepackage{algorithm}
\usepackage{algorithmic}
\usepackage{amssymb, amsmath, amsthm}
\usepackage{url}            
\usepackage{booktabs}       
\usepackage{amsfonts}       
\usepackage{nicefrac}       
\usepackage{microtype}      
\usepackage{color}
\usepackage{graphicx}
\usepackage{mathtools}
\usepackage{enumitem}
\usepackage{thmtools}
\usepackage{thm-restate}
\usepackage{enumitem}
\usepackage[round]{natbib}
\usepackage[T1]{fontenc}
\usepackage{lmodern}

\newcommand{\sqnorm}[1]{\left\|#1\right\|^2}

\newcommand{\vectornorm}[1]{\left\|#1\right\|}
\newcommand{\norm}[1]{{}\left\| #1 \right\|}
\newcommand{\bigO}{\mathcal{O}}

\newcommand{\g}{{\bf g}}

\newcommand{\h}{{\bf h}}

\newcommand{\s}{{\bf s}}

\renewcommand{\u}{{\bf u}}

\newcommand{\vdelta}{{\bf \delta}}

\newcommand{\w}{{\bf w}}

\newcommand{\x}{{\bf x}}
\newcommand{\y}{{\bf y}}
\newcommand{\z}{{\bf z}}

\def\Am{{\bf A}}
\def\Bm{{\bf B}}

\def\Hm{{\bf H}}
\def\Im{{\bf I}}

\def\Tm{{\bf T}}

\def\Xm{{\bf X}}

\newcommand{\A}{{\mathcal A}}

\newcommand{\E}{{\mathbb E}}
\newcommand{\F}{{\mathcal F}}
\newcommand{\K}{{\mathcal K}}
\newcommand{\M}{{\mathcal M}}

\newcommand{\R}{{\mathbb{R}}}
\renewcommand{\S}{{\mathcal S}}
\newcommand{\T}{{\mathbb T}}
\newcommand{\U}{{\mathcal U}}
\newcommand{\X}{{\mathcal X}}
\newcommand{\Y}{{\mathcal Y}}
\newcommand{\Z}{{\mathbf Z}}

\newtheorem{theorem}{Theorem}
\newtheorem{lemma}[theorem]{Lemma}

\newtheorem{assumption}[theorem]{Assumption}

\newtheorem{condition}{Condition}
\newtheorem{remark}{Remark}

\DeclarePairedDelimiter\ceil{\lceil}{\rceil}

\newcommand*{\nfrac}[2]{#1\left/\vphantom{#1}#2\right.}
\newcommand{\argmin}{argmin}

\newcommand{\methodname}{STM}

\title{A Sub-sampled Tensor Method for Non-convex Optimization \\ \small{Initial title: 
A Stochastic Tensor Method for Non-convex Optimization}}

\date{\today \\ First version: November 29, 2019}

\author{Aurelien Lucchi \quad Jonas Kohler \\
Department of Computer Science, ETH Z\"urich}

\begin{document}
\maketitle

\begin{abstract}
\noindent We present a stochastic optimization method that uses a fourth-order regularized model to find local minima of smooth and potentially non-convex objective functions with a finite-sum structure. This algorithm uses sub-sampled derivatives instead of exact quantities. The proposed approach is shown to find an $(\epsilon_1,\epsilon_2,\epsilon_3)$-third-order critical point in at most $\bigO\left(\max\left(\epsilon_1^{-4/3}, \epsilon_2^{-2}, \epsilon_3^{-4}\right)\right)$ iterations, thereby matching the rate of deterministic approaches. In order to prove this result, we derive a novel tensor concentration inequality for sums of tensors of any order that makes explicit use of the finite-sum structure of the objective function.
\end{abstract}


\section{Introduction}

We consider the problem of optimizing an objective function of the form
\begin{equation}
\x^* = \arg \min_{\x \in \R^d}  \left [ f(\x) := \frac1n \sum_{i=1}^n f_i(\x) \right ],
\label{eq:f_x}
\end{equation}
where $f(\x) \in C^3(\R^d, \R)$ is a not necessarily convex loss function defined over $n$ datapoints.

Our setting is one where access to the exact function gradient $\nabla f$ is computationally expensive (e.g. large-scale setting where $n$ is large) and one therefore wants to access only stochastic evaluations $\nabla f_i$, potentially over a mini-batch. In such settings, stochastic gradient descent (SGD) has long been the method of choice in the field of machine learning. Despite the uncontested empirical success of SGD to solve difficult optimization problems -- including training deep neural networks -- the convergence speed of SGD is known to slow down close to saddle points or in ill-conditioned landscapes~\citep{nesterov2004introductory, dauphin2014identifying}. While gradient descent requires $\bigO(\epsilon^{-2})$ oracle evaluations\footnote{In "oracle evaluations" we include the number of function and gradient evaluations as well as evaluations of higher-order derivatives.} to reach an $\epsilon$-approximate \emph{first-order} critical point, the complexity worsens to $\bigO(\epsilon^{-4})$ for SGD. In order to recover the rate of the deterministic gradient descent, a common technique is to rely on adaptive sampling techniques~\citep{friedlander2012hybrid} or variance reduction~\citep{johnson2013}.
Another way to speed up the convergence of gradient-based methods is to rely on higher-order derivatives. For instance regularized Newton methods and trust-region methods exploit curvature information, allowing them to enjoy faster convergence to a \emph{second-order} critical point.
In this work, we focus our attention on regularized high-order methods to optimize Eq.~\eqref{eq:f_x}, which construct and optimize a local Taylor model of the objective in each iteration with an additional step length penalty term that depends on how well the model approximates the real objective. This paradigm goes back to Trust-Region and Cubic Regularization methods, which also make use of regularized models to compute their update step~\citep{conn2000trust,nesterov2006cubic, cartis2011adaptive}. For the class of second-order Lipschitz smooth functions,~\cite{nesterov2006cubic} showed that the Cubic Regularization framework finds an $(\epsilon_1, \epsilon_2)$-approximate second-order critical point in at most $\max \left(\bigO(\epsilon_1^{-3/2}), \bigO(\epsilon_2^{-3}) \right)$ iterations, thus achieving the best possible rate in this setting~\citep{carmon2017lower}. Recently, stochastic extensions of these methods have appeared in the literature such as~\cite{cartis2015global, kohler2017sub, tripuraneni2018stochastic, xu2017newton}. These will be discussed in further details in Section~\ref{sec:related_work}.

Since the use of second derivatives can provide significant theoretical speed-ups, a natural question is whether higher-order derivatives can result in further improvements. This questions was answered affirmatively in~\cite{birgin2017worst} who showed that using derivatives up to order $p \geq 1$ allows convergence to an $\epsilon_1$-approximate first-order critical point in at most $\bigO(\epsilon_1^{-(p+1)/p})$ evaluations. This result was extended to $\epsilon_2$-second-order stationarity in~\cite{cartis2020concise}, which proves an $\bigO(\epsilon_2^{-(p+1)/(p-1)})$ rate. Yet, these results assume a deterministic setting where access to exact evaluations of the function derivatives is needed and -- to the best of our knowledge -- the question of using high-order ($p \geq 3$) derivatives in a stochastic setting has received little consideration in the literature so far. We focus our attention on the case of computing derivative information of up to order $p=3$. It has recently been shown in~\cite{nesterov2015implementable} that, while optimizing degree four polynomials is NP-hard in general~\cite{hillar2013most}, the specific models that arise from a third-order Taylor expansion with a quartic regularizer can still be optimized efficiently.

From an application point of view, second-order methods (case $p=2$) are potentially of interest for problems that are ill-conditioned. However, the main drawback of these methods if their high computational cost which has so far impeded their adoption in the field of machine learning and others. Some recent work has shown that sub-sampling~\citep{erdogdu2015convergence} or other dimensionality reduction techniques~\citep{pilanci2017newton} provide some solutions to reduce the computational cost. We refer the reader to Section~\ref{sec:related_work} for further details. In general, higher-order methods ($p > 2$) become especially of interest when highly-accurate solutions or high-order optimality conditions are needed. We refer the reader to~\citep{floudas2005global, gould2017higher}, among others, for an overview of potential applications.

The main contribution of this work is to demonstrate that a sub-sampled third-order regularized method under appropriate sampling conditions can find an $(\epsilon_1, \epsilon_2, \epsilon_3)$ third-order stationary point in at most $\max \left(\bigO(\epsilon_1^{-4/3}), \bigO(\epsilon_2^{-2}), \bigO(\epsilon_3^{-4}) \right)$ iterations (Thm.~\ref{th:worst_case_complexity}). This results in an algorithm that can adaptively change the batch-size in order to match the results obtained by deterministic methods. In order to prove this result, we develop a novel tensor concentration inequality (Thm.~\ref{th:tensor_hoefdding_no_replacement}) for sums of tensors of any order and make explicit use of the finite-sum structure given in Eq.~\eqref{eq:f_x}. Together with existing matrix and vector concentration bounds~\citep{tropp2015introduction}, this allows us to define the sufficient amount of samples needed for convergence. We thereby provide theoretically motivated sampling schemes for the derivatives of the objective for both sampling with and without replacement.


\section{Related work}
\label{sec:related_work}

\paragraph{Sampling techniques for first-order methods.}
In large-scale learning ($n \gg d$) most of the computational cost of traditional deterministic optimization methods is spent on computing the exact gradient information. A common technique to address this issue is to use sub-sampling to compute an unbiased estimate of the gradient. The simplest instance is SGD whose convergence does not depend on the number of datapoints $n$. However, the variance in the stochastic gradient estimates slows its convergence down. The work of~\cite{friedlander2012hybrid} explored a sub-sampling technique in the case of convex functions, showing that it is possible to maintain the same convergence rate as full-gradient descent by carefully increasing the sample size over time. Another way to recover a linear rate of convergence for strongly-convex functions is to use variance reduction~\citep{johnson2013, defazio2014saga, roux2012stochastic, hofmann2015variance, daneshmand2016small}. The convergence of SGD and its variance-reduced counterpart has also been extended to non-convex functions~\cite{ghadimi2013stochastic, reddi2016stochastic} but the guarantees these methods provide are only in terms of first-order stationarity. However, the work of~\cite{ge2015escaping, sun2015nonconvex, daneshmand2018escaping} among others showed that SGD can achieve stronger guarantees in the case of strict-saddle functions. Yet, the convergence rate has a polynomial dependency to the dimension $d$ and the smallest eigenvalue of the Hessian which can make this method fairly impractical.

\paragraph{Second-order methods.}
For second-order methods that are not regularized or that make use of positive definite Hessian approximations (e.g. Gauss-Newton), the problem of avoiding saddle points is even worse as they might be attracted by saddle points or even local maximima~\citep{dauphin2014identifying}. Another predominant issue is the computation (and storage) of the Hessian matrix, which can be partially addressed by Quasi-Newton methods such as (L-)BFGS. An increasingly popular alternative is to use sub-sampling techniques to approximate the Hessian matrix, such as in~\cite{byrd2011use} and \cite{erdogdu2015convergence}. The latter method uses a low-rank approximation of the Hessian to reduce the complexity per iteration. However, this yields a composite convergence rate: quadratic at first but only linear near the minimizer.

Finally, an alternative to sub-sampling is sketching~\citep{pilanci2017newton} where the Hessian matrix is approximated via a random projection. We refer the reader to~\citep{berahas2020investigation} for a comparison of sub-sampling and sketching techniques. In brief, sub-sampling is seen as a simple method to implement and inexpensive, while sketching has some additional computational cost that might only be beneficial in some settings, for instance when the individual components of the loss (in a finite-sum setting) are highly dissimilar, which would require large batches in a sub-sampling setting.

\paragraph{Cubic regularization and trust region methods.}
Trust region methods are among the most effective algorithmic frameworks to avoid pitfalls such as local saddle points in non-convex optimization. Classical versions iteratively construct a local quadratic model and minimize it within a certain radius wherein the model is trusted to be sufficiently similar to the actual objective function. This is equivalent to minimizing the model function with a suitable \textit{quadratic} penalty term on the stepsize. Thus, a natural extension is the cubic regularization method introduced by~\cite{nesterov2006cubic} that uses a \textit{cubic} over-estimator of the objective function as a regularization technique for the computation of a step to minimize the objective function. The drawback of their method is that it requires computing the exact minimizer of the cubic model, thus requiring the exact gradient and Hessian matrix. However, finding a global minimizer of the cubic model $m_k(\s)$ may not be essential in practice and doing so might be prohibitively expensive from a computational point of view. ~\cite{cartis2011adaptive} introduced a method named ARC which relaxed this requirement by letting $\s_k = \argmin_\s m_k(\s)$ be an approximation to the minimizer. The model defined by the adaptive cubic regularization method introduced two further changes. First, instead of computing the exact Hessian $\Hm_k$ it allows for a symmetric approximation $\Bm_k$. Second, it introduces a dynamic step-length penalty parameter $\sigma_k$ instead of using the global Lipschitz constant. Our approach relies on the same adaptive framework.

There have been efforts to further reduce the computational complexity of optimizing the model. For example, \cite{agarwal2016finding} refined the approach of~\cite{nesterov2006cubic} to return an approximate local minimum in time which is linear in the input. Similar improvements have been made by ~\cite{carmon2016gradient} and ~\cite{hazan2016linear}. These methods provide alternatives to minimize the cubic model and can thus be seen as complementary to our approach. Finally,~\cite{blanchet2016convergence} proposed a stochastic trust region method but their analysis does not specify any accuracy level required for the estimation of the stochastic Hessian. ~\cite{cartis2015global} also analyzed a probabilistic cubic regularization variant that allows for approximate second-order models. ~\cite{kohler2017sub} provided an explicit derivation of sampling conditions to preserve the worst-case complexity of ARC. Other works also derived similar stochastic extensions to cubic regularization, including~\cite{xu2017newton} and~\cite{tripuraneni2018stochastic}. The worst-case rate derived in the latter includes the complexity of a specific model solver introduced in~\cite{carmon2016gradient}.

\paragraph{High-order models.}
A hybrid algorithm suggested in~\cite{anandkumar2016efficient} adds occasional third-order steps to a cubic regularization method, thereby obtaining provable convergence to some type of third-order local minima. A recent extension with an adaptive parameter tuning strategy was proposed by~\cite{zhu2020adaptive}. Both approaches rely on deterministic derivatives, while we address the problem of stochastic optimization.
High-order derivatives can also directly be applied within the regularized Newton framework, as done e.g. by ~\cite{birgin2017worst} that extended cubic regularization to a $p$-th high-order model (Taylor approximation of order $p$) with a $(p+1)$-th order regularization, proving iteration complexities of order $\bigO(\epsilon^{-\nfrac{(p+1)}{p}})$ for first-order stationarity. 
These convergence guarantees are extended to the case of inexact functions and derivatives by~\cite{bellavia2018adaptive} who present a dynamic strategy to deal with inaccurate functions and derivatives.
However, their work focuses on a more abstract setting and we do not see any direct way to implement such a strategy in the finite-sum setting. Furthermore, they only discuss a possible implementation for the case where $p=2$ while we focus on the case $p=3$. In order to address the latter case, we derive a new tensor concentration bound to give a condition on the number of samples required to compute the derivatives (see Theorems~\ref{th:tensor_hoefdding_no_replacement} and~\ref{th:tensor_hoefdding}). We believe this new concentration result itself is of independent interest.

The convergence guarantees of $p$-th order models are extended to second-order stationarity by~\cite{cartis2020concise}. Notably, all these approaches require optimizing a $(p+1)$-th order polynomial, which is known to be a difficult problem. Recently,~\cite{nesterov2015implementable, grapiglia2019inexact} introduced an implementable method for the case $p=3$ in the deterministic case and for convex functions.
In this work, we too focus on the case $p=3$ for which we will provide an implementable version of our algorithm, along with an experimental validation on several datasets. We however note that the analysis could probably be extended to the case of arbitrary $p$.

Finally, another line of works~\citep{allen2018natasha, xu2018first} considers methods that do not use high-order derivatives explicitly within a regularized Newton framework but rather rely on other routines -- such as Oja's algorithm -- to explicitly find negative curvature directions and couple those with SGD steps.


\section{Formulation}

\subsection{Notation \& Assumptions}

First, we lay out some standard assumptions regarding the function $f$ as well as the required approximation quality of the high-order derivatives.

\begin{assumption}[Continuity]
\label{a:continuity}
The functions $f_i \in C^3(\R^d, \R)$, $\nabla f_i, \nabla^2 f_i$, $\nabla^3 f_i$ are Lipschitz continuous for all $i$, with Lipschitz constants $L_f, L_g, L_b$ and $L_t$ respectively.
\end{assumption}

In the following, we will denote the $p$-th directional derivative of the function $f$ at $\x$ along the directions $\h_j \in \R^d, j = 1 \dots p$ as
\begin{equation}
\nabla^p f(\x)[\h_1, \dots \h_p].
\end{equation}
For instance, $\nabla f(\x)[\h] = \nabla f(\x)^\top \h$ and $\nabla^2 f(\x)[\h]^2 = \h^\top \nabla^2 f(\x) \h$.

Assumption~\ref{a:continuity} implies that for each $p = 0 \dots 3$,
\begin{equation}
\| \nabla^p f_i (\x) - \nabla^p f_i (\y) \|_{[p]} \leq L_p \| \x - \y \|
\label{eq:Lipschitz_continuous}
\end{equation}
for all $\x, \y \in \R^d$ and where $L_0 = L_f, L_1 = L_g, L_2 = L_b, L_3 = L_t$.

As in~\cite{cartis2020concise}, $\| \cdot \|_{[p]}$ is the tensor norm recursively induced by the Euclidean norm $\| \cdot \|$ on the space of $p$-th order tensors.

\subsection{Sub-sampled surrogate model}

We construct a surrogate model to optimize $f$ based on a truncated Taylor approximation as well as a power prox function weighted by a sequence $\{ \sigma_k \}_k$ that is controlled adaptively according to the fit of the model to the function $f$. Since the full Taylor expansion of $f$ requires computing high-order derivatives that are expensive, we instead use an inexact model defined as
\begin{align}
m_k(\s) &= \phi_k(\s) + \frac{\sigma_k}{4} \vectornorm{\s}^4_2, \nonumber \\
\phi_k(\s) &= f(\x_k) + \g_k^\top \s + \frac12 \s^\top \Bm_k \s + \frac{1}{6} \Tm_k [\s]^3
\label{eq:model}
\end{align}
where $\g_k, \Bm_k$ and $\Tm_k$ approximate the derivatives $\nabla f(\x_k), \nabla^2 f(\x_k)$ and $\nabla^3 f(\x_k)$ through sampling as follows. Three sample sets $\S^g, \S^b$ and $\S^t$ are drawn and the derivatives are then estimated as
\begin{align}
\g_k &= \frac{1}{|\S^g|} \sum_{i \in \S^g} \nabla f_i(\x_k),
\Bm_k = \frac{1}{|\S^b|} \sum_{i \in \S^b} \nabla^2 f_i(\x_k), \nonumber \\
\Tm_k &= \frac{1}{|\S^t|} \sum_{i \in \S^t} \nabla^3 f_i(\x_k).
\label{eq:sampled_derivatives}
\end{align}

\paragraph{Model derivatives}

The first derivative of the model w.r.t. $\s$ is
\begin{equation}
\nabla_\s m_k(\s) = \g_k + \Bm_k \s + \frac{1}{2} \Tm_k [\s]^2 + \sigma_k \s \| \s \|^2.
\end{equation}

For the second-order derivative $\nabla_\s^2 m_k(\s)$, we get:
\begin{equation}
\nabla^2_\s m_k(\s) = \Bm_k + \Tm_k [\s] + \frac{\sigma_k}{4} \nabla^2_\s \| \s \|^4,
\end{equation}
where
\begin{equation}
\frac14 \nabla^2_\s \| \s \|^4 = \nabla_\s \; \s \| \s \|^2 = \| \s \|^2 \Im + 2 \s \s^\top \succcurlyeq \| \s \|^2.
\label{eq:hessian_s}
\end{equation}


\paragraph{Sampling conditions}

We will make use of the following condition in order to reach an $\epsilon$-critical point:


\begin{condition}
For a given $\epsilon$ accuracy, one can choose the size of the sample sets $\S^g, \S^b, \S^t$ for sufficiently small $\kappa_g, \kappa_b, \kappa_t > 0$ such that:
\begin{align}
\label{eq:sampling_g}
\| \g_k - \nabla f(\x_k) \| &\leq \kappa_g \epsilon \\ 
\label{eq:sampling_b}
\| (\Bm_k - \nabla^2 f(\x_k)) \s \| &\leq \kappa_b \epsilon^{2/3} \| \s \|, \;\; \forall \s \in \R^d \\
\label{eq:sampling_t}
\| \Tm_k[\s]^2 - \nabla^3 f(\x_k) [\s]^2 \| &\leq \kappa_t \epsilon^{1/3} \| \s \|^2, \;\; \forall \s \in \R^d.
\end{align}
\label{cond:sampling}
\end{condition}

In Lemma~\ref{lemma:sampling_conditions}, we prove that we can choose the size of each sample set $\S^g, \S^b$ and $\S^t$ to satisfy the conditions above, without requiring knowledge of the length of the step $\| \s_k \|$. We will present a  convergence analysis of \methodname, proving that the convergence properties of the deterministic methods~\citep{birgin2017worst, nesterov2015implementable} can be retained by a sub-sampled version at the price of slightly worse constants.

\subsection{Algorithm}
\label{sec:algo}

 \begin{algorithm*}[tb]
   \caption{Sub-sampled Tensor Method (STM)}
   \label{alg:stm}
\begin{algorithmic}[1]
   \STATE {\bfseries Input:} \\ 
   $\quad$ Starting point $\x_0 \in \R^d$ (e.g~$\x_0 = {\bf 0}$) \\
   $\quad 0 < \gamma_1 < 1 < \gamma_2 < \gamma_3, 1>\eta_2>\eta_1>0$, and $\sigma_0>0, \sigma_{min}>0$
   \FOR{$k=0,1,\dots,\text{until convergence}$}
   \STATE Sample gradient $\g_k$, Hessian $\Bm_k$ and $\Tm_k$ such that Eq.~\eqref{eq:sampling_g}, Eq.~\eqref{eq:sampling_b} \& Eq.~\eqref{eq:sampling_t} hold.
   \STATE Obtain $\s_k$ by solving $m_k(\s_k)$ (Eq.~\eqref{eq:model}) such that Condition~\ref{cond:approximate_min} holds.
   \STATE Compute $f(\x_k + \s_k)$ and 
	\begin{equation}
	\rho_k=\dfrac{f(\x_k)-f(\x_k + \s_k)}{f(\x_k) - \phi_k(\s_k)}.
	\end{equation}
	\STATE Set
	\begin{equation}
	\x_{k+1} = \begin{cases}
	\x_k + \s_k & \text{ if } \rho_k \geq \eta_1\\
	\x_k & \text{ otherwise.}
	\end{cases}
	\end{equation}
	\STATE Set
	\begin{equation} \label{eq:sigma_update}
	\sigma_{k+1}= \begin{cases}
	[\max\{\sigma_{min},\gamma_1 \sigma_k \}, \sigma_k] & \text{ if } \rho_k>\eta_2 \text{ (very successful iteration)}\\
	[\sigma_k, \gamma_2 \sigma_k] & \text{ if } \eta_2\geq\rho_k\geq \eta_1 \text{ (successful iteration)}\\
	[\gamma_2 \sigma_k, \gamma_3 \sigma_k] & \text{ otherwise}\text{ (unsuccessful iteration)}.
	\end{cases}
	\end{equation}
   \ENDFOR
\end{algorithmic}
\end{algorithm*}

The optimization algorithm we consider is detailed in Algorithm~\ref{alg:stm}. A deterministic version of this algorithm is presented in~\citep{cartis2022evaluation} along with a complexity analysis.
The major difference with Algorithm~\ref{alg:stm} is that, at iteration step $k$, we sample three sets of datapoints from which we compute stochastic estimates of the derivatives of $f$ so as to satisfy Condition~\ref{cond:sampling}.
We then obtain the step $\s_k$ by solving the problem 
\begin{equation}
\s_k = \arg\min_{\s \in \R^d} m_k(\s),
\label{eq:subproblem}
\end{equation}
either exactly or approximately (details will follow shortly) and update the regularization parameter $\sigma_k$ depending on $\rho_k$, which measures how well the model approximates the real objective. This is accomplished by differentiating between different types of iterations. Successful iterations (for which $\rho_k \geq \eta_1$) indicate that the model is, at least locally, an adequate approximation of the objective such that the penalty parameter is decreased in order to allow for longer steps. We denote the index set of all successful iterations between 0 and $k$ by $\S_k = \{ 0 \leq j \leq k | \rho_j \geq \eta_1 \}$. We also denote by $\U_k$ its complement in $\{ 0, \dots k \}$  which corresponds to the index set of unsuccessful iterations.

\paragraph{Exact model minimization}
Solving Eq.~\eqref{eq:model} requires minimizing a nonconvex multivariate polynomial. As pointed out in~\cite{nesterov2015implementable, baes2009estimate}, this problem is computationally expensive to solve in general and further research is needed to establish whether one could design a practical minimization method. In~\cite{nesterov2015implementable}, the authors demonstrated that an appropriately regularized Taylor approximation of convex functions is a convex multivariate polynomial, which can be solved using the framework of relatively smooth functions developed in~\cite{lu2018relatively}. As long as the involved models are convex, this solver could be used in Algorithm \ref{alg:stm} but it is unclear how to generalize this method to the non-convex case. Fortunately, we will see next that exact model minimizers are not even needed to establish global convergence guarantees for our method.

\paragraph{Approximate model minimization}

Exact minimization of the model in Eq.~\eqref{eq:model} is often computationally expensive, especially given that it is required for every parameter update in Algorithm~\ref{alg:stm}. In the following, we explore an approach to approximately minimize the model while retaining the convergence guarantees of the exact minimization. Instead of requiring the exact optimality conditions to hold, we use weaker conditions that were also used in prior work~\citep{birgin2017worst, cartis2020concise}. First, we define three criticality measures based on first-, second-, and third-order information:
\begin{align}
    \chi_{f, 1}(\x_k) &:= \| \nabla f(\x_k) \|, \nonumber \\
    \chi_{f, 2}(\x_k) &:= \max \left(0, -\lambda_{min}(\nabla^2 f(\x_k)) \right), \nonumber \\
    \chi_{f, 3}(\x_k) &:= \max_{\y \in \Z_{k+1}} \left| \nabla^3 f(\x_k)[\y]^3  \right|,
\label{eq:def_chi_f}
\end{align}
where $\lambda_{min}(\nabla^2 f(\x))$ is the minimum eigenvalue of the Hessian matrix $\nabla^2 f(\x)$.
The set $\Z_{k+1}$ is used to capture third-order optimality and is typically defined as the kernel of $\nabla^2 f(\x_k)$~\cite{cartis2018second}, or alternatively using a $\zeta$-approximate notion of optimality defined in~\cite{cartis2022evaluation} as
\begin{equation}
\Z_{k+1} := \{ \y \mid \| \y \| = 1 \text{ and } | \nabla^2 f(\x_k)[\y]^2 | \leq \zeta \}.
\end{equation}
We refer the reader to~\cite{cartis2018second, cartis2022evaluation} for an in-depth discussion about the difficulties of finding exact high-order minimizers, which go beyond the scope of this work.

The same criticality measures are defined for the model $m_k(\s)$,
\begin{align}
    \chi_{m, 1}(\x_k, \s) := \| \nabla_\s m_k(\s) \|, \nonumber \\
    \chi_{m, 2}(\x_k, \s) := \max \left(0, -\lambda_{min}(\nabla_\s^2 m_k(\s)) \right), \nonumber \\
    \chi_{m, 3}(\x_k, \s) := \max_{\y \in \M_{k+1}} \left| \nabla^3 m_k(\s)[\y]^3  \right|,
\label{eq:def_chi_m}
\end{align}
where
\begin{equation}
\M_{k+1} := \{ \y \mid \| \y \| = 1 \text{ and } | \nabla^2 m_k(\s)[\y]^2 | \leq \zeta \}.
\end{equation}

Finally, we state the approximate optimality condition required to find the step $\s_k$.
\begin{condition}
For each iteration $k$, the step $\s_k$ is computed so as to approximately minimize the model $m_k(\s_k)$ in the sense that the following conditions hold:
\begin{align}
m_k(\s_k) &< m_k({\bf 0}) \nonumber \\
\chi_{m,i}(\x_k, \s_k) &\leq  \theta \| \s_k \|^{4-i}, \quad \theta > 0, \; \text{ for } i=1,\ldots, 3.
\label{eq:termination_criterion}
\end{align}
\label{cond:approximate_min}
\end{condition}

\paragraph{Termination criterion}
The algorithm should be stopped once we reach a third-order stationary point, i.e. a point $\x^*$ such that $\chi_{f,i}(\x^*) \leq \epsilon_i \quad \text{ for } i=1, \dots, 3$. In the setting where we only have access to sub-sampled derivatives, we replace this condition by the approximate quantity $\chi_{m,i} \leq \epsilon_i$, which changes to the final accuracy to
\begin{equation}
\chi_{f,i}(\x_k) = \chi_{f,i}(\x_k) \pm \chi_{m, i}(\x_k, \s) 
\leq \chi_{m, i}(\x_k, \s) + | \chi_{f,i}(\x_k) - \chi_{m, i}(\x_k, \s) | 
\leq (1 + \kappa_i) \epsilon_i,
\end{equation}
where the last inequality is due to the reverse triangle inequality.

\section{Fulfilling the sampling conditions}

In this section, we show how to ensure that the sampling conditions in Eqs.~\eqref{eq:sampling_g}-\eqref{eq:sampling_t} are satisfied. We discuss two cases: i) random sampling \emph{without} replacement, and ii) sampling with replacement.

\subsection{Sampling without replacement}

We first show that one can use random sampling without replacement and choose the size of the sample sets in order to satisfy Condition~\ref{cond:sampling} with high probability.
First, we need to develop a new concentration bound for tensors based on the spectral norm. Existing tensor concentration bounds are not applicable to our setting. Indeed, ~\cite{luo2019bernstein} relies on a different norm which can not be translated to the spectral norm required in our analysis while the bound derived in~\cite{vershynin2019concentration} relies on a specific form of the input tensor.

Formally, let $(\Omega, \F, P)$ be a probability space and let $\X$ be a real $(m, d)$ random tensor, i.e. a measurable map from $\Omega$ to $\T_{m,d}$ (the space of real tensors of order $m$ and dimension $d$).

Our goal is to derive a concentration bound for a sum of $n$ identically distributed tensors sampled \emph{without} replacement, i.e. we consider
$$
\X = \sum_{i=1}^n \Y_i,
$$
where each tensor $\Y_i$ is sampled from a population $\A$ of size $N>n$.

The concentration result derived in this section is based on an $\epsilon$-net argument for sums of tensors and is inspired by the proof introduced in~\cite{tomioka2014spectral}. Formally, we consider a tensor $\X \in \R^{d_1 \times \dots \times d_k}$ of order $k$ whose spectral norm is defined as
\begin{equation}
\| \X \| = \sup_{\substack{\u_1, \dots, \u_k\\ \| \u_i \|=1}} \X(\u_1, \dots \u_k).
\label{eq:spectralnorm}
\end{equation}

Note that for symmetric tensors, the spectral norm is simply equal to $\| \X \| = \sup_{\u \mid \| \u \|=1} \X(\u, \dots \u)$~\citep{anandkumar2016efficient}.

\paragraph{Auxiliary results}

We will need the following results in order to complete the proof of the main results presented in this section.

\begin{theorem}[Matrix Hoeffding-Serfling Inequality~\cite{wang2018stochastic}]
\label{th:Matrix_Hoeffding_Serfling_Inequality}
Let $\mathcal{A}:= \{\Am_1, \cdots, \Am_N\}$ be a collection of real-valued matrices in $\R^{d_1\times d_2}$ with bounded spectral norm, i.e.,  $\| \Am_i \| \leqslant\sigma$ for all $i = 1, \ldots, N$ and some $\sigma > 0$. 
Let $\Xm_1, \cdots, \Xm_n$ be $n<N$ samples from $\mathcal{A}$ under the sampling without replacement. Denote $\mu := \frac{1}{N} \sum_{i = 1}^{N} \Am_i$. Then,  for any $t > 0$, 
\begin{align*}
P \bigg(\bigg\|  \frac{1}{n}\sum_{i=1}^{n}  \Xm_i -   \mu\bigg\|  \geqslant   t \bigg)  \leq    (d_1 + d_2) \exp \bigg( -  \frac{n t^2}{8 \sigma^2 (1+1/n) (1- n/N)}\bigg).
\end{align*}
\end{theorem}

\begin{lemma}[~\cite{bardenet2015concentration}]
\label{lemma:Hoeffding_Serfling_lemma}
Let $\A := \{ y_1, \dots y_N \}$ be a finite population of $N$ points in $\mathbb{R}$ and $Y_1, \dots, Y_n$ be a random sample drawn without replacement from $\A$. Define $X_n = \frac{1}{n}\sum_{k=1}^n (Y_k - \mu)$ where $\mu= \frac {1}{N}\sum_{i=1}^N y_i$. Assume that $a \leq Y_k \leq b$, then or any $s > 0$, it holds that
\begin{eqnarray*}
\log \E \exp ( s X_n ) \leq \frac{(b-a)^2}{8} \frac{s^2}{n^2} (n+1) \biggl(1 - \frac{n}{N} \biggr) .
\end{eqnarray*}
\end{lemma}

\paragraph{Proof idea} In the following, we first provide a concentration bound for a fixed set of unit-length vectors $\u_1,\ldots,\u_k$. (Lemma~\ref{lem:entry_concentration_no_replacement}). We then extend this result to arbitrary vectors on the unit sphere in order to obtain a concentration bound for the tensor $\X$ by using a covering argument similar to~\cite{tomioka2014spectral}.\\

\begin{restatable}{lemma}{concnoreplacement}
\label{lem:entry_concentration_no_replacement}
Let $\X$ be a sum of $n$ i.i.d. tensors $\Y_i \in \R^{d_1 \times \dots \times d_k}$ sampled without replacement from a finite population $\A$ of size $N$. Consider a fixed set of vectors $\u_1, \dots \u_k$ such that $\| \u_i \| = 1$ and assume that for each tensor $i$, $a \leq \Y_i(\u_1,\ldots,\u_k) \leq b$. Let $\sigma := (b-a)$, then we have
\begin{align*}
P\left(|\X(\u_1,\ldots,\u_k) - \E[\X(\u_1,\ldots,\u_k)] | \geq t \right) \leq 2 \exp \left( - \frac{2t^2 n^2}{ \sigma^2 (n+1) (1 - n/N)} \right).
\end{align*}
\end{restatable}
\begin{proof}

By Markov's inequality and Hoeffding's lemma, we have
\begin{align*}
P&\left(\X(\u_1,\ldots,\u_k) - \E[\X(\u_1,\ldots,\u_k)] \geq t\right) \\
&= P\left(e^{s(\X(\u_1,\ldots,\u_k)- \E[\X(\u_1,\ldots,\u_k)])}\geq e^{st} \right) \\
&\leq e^{-st} \E\left[e^{(s(\X(\u_1,\ldots,\u_k)- \E[\X(\u_1,\ldots,\u_k)]}) \right] \\
&= e^{-st} \E\left[e^{(s(\sum_i \Y_i(\u_1,\ldots,\u_k)- \E[\sum_i \Y_i(\u_1,\ldots,\u_k)]}) \right] \\
&\stackrel{(i)}{\leq} \exp\left(-st + \frac{\sigma^2}{8} \frac{s^2}{n^2} (n+1) \biggl(1 - \frac{n}{N} \biggr)\right),
\end{align*}
where $(i)$ follows from Lemma~\ref{lemma:Hoeffding_Serfling_lemma}.

After minimizing over $s$, we obtain
\begin{equation*}
P\left(\X(\u_1,\ldots,\u_k) - \E[\X(\u_1,\ldots,\u_k)] \geq t\right) \leq \exp \left( - \frac{2t^2 n^2}{ \sigma^2 (n+1) (1 - n/N)} \right).
\end{equation*}

By symmetry, one can easily show that $$P(\X(\u_1,\ldots,\u_k) \leq -t)\leq \exp \left( - \frac{2t^2 n^2}{ \sigma^2 (n+1) (1 - n/N)} \right).$$
We then complete the proof by taking the union of both cases.
\end{proof}

We are now ready to prove a new concentration inequality for sums of tensors.

\begin{restatable}{theorem}{tensorhoeffdingnoreplacement}[Tensor Hoeffding-Serfling Inequality]
\label{th:tensor_hoefdding_no_replacement}
Let $\X$ be a sum of $n$ tensors $\Y_i \in \R^{d_1 \times \dots \times d_k}$ sampled without replacement from a finite population $\A$ of size $N$. Let $\u_1, \dots \u_k$ be such that $\| \u_i \| = 1$ and assume that for each tensor $i$, $a \leq \Y_i(\u_1,\ldots,\u_k) \leq b$. Let $\sigma := (b-a)$, then we have
\begin{equation*}
P(\norm{\X - \E \X} \geq t) \leq k_0^{(\sum_{i=1}^{k} d_i)} \cdot 2 \exp \left( - \frac{t^2 n^2}{2 \sigma^2 (n+1) (1 - n/N)} \right),
\end{equation*}
where $k_0 = \left(\frac{2k}{\log(3/2)}\right)$.
\end{restatable}
\begin{proof}
The main idea is to create an $\epsilon$-net of countable size to cover the space $S^{d_1-1}, \ldots, S^{d_k-1}$. Formally, let $C_1,\ldots,C_k$ be $\epsilon$-covers of $S^{d_1-1}, \ldots, S^{d_k-1}$. Then since $S^{d_1-1} \times \cdots \times S^{d_k-1}$ is compact, there exists a maximizer $(\u_1^\ast,\ldots,\u_k^\ast)$ of \eqref{eq:spectralnorm}. Using the $\epsilon$-covers, we have
\begin{equation*}
\norm{\X} = \X(\bar{\u}_1+\vdelta_1,\ldots,\bar{\u}_k+\vdelta_k),
\end{equation*}
where $\bar{\u}_i \in C_i$ and $\|\vdelta_i\| \leq \epsilon$ for $i=1,\ldots,k$.
 
Now
\begin{equation*}
\norm{\X} \leq \X(\bar{\u}_1,\ldots,\bar{\u}_k) + \left(\epsilon k + \epsilon^2\binom{k}{2}+\cdots \epsilon^k\binom{k}{k}\right) \norm{\X}.
\end{equation*}

Take $\epsilon=\frac{\log(3/2)}{k}$ then the sum inside the parenthesis
 can be bounded as follows:
\begin{equation*}
\epsilon k + \epsilon^2\binom{k}{2} + \cdots \epsilon^k\binom{k}{k} \leq
\epsilon k + \frac{(\epsilon k)^2}{2!} + \cdots \frac{(\epsilon k)^k}{k!} \leq e^{\epsilon k} - 1 = \frac12.
\end{equation*}

Thus we have 
\begin{equation*}
\norm{\X} \leq 2 \max_{\bar{\u}_1\in C_1,\ldots,\bar{\u}_k\in C_k} \X(\bar{\u}_1,\ldots,\bar{\u}_k).
\end{equation*}

We conclude the proof with one last step that is a simple adaptation of the proof in~\cite{tomioka2014spectral}, combined with the result of Lemma~\ref{lem:entry_concentration_no_replacement}.

Since the $\epsilon$-covering number $|C_k|$ can be bounded by $\epsilon/2$-packing number, which can be bounded by $(2/\epsilon)^{d_k}$, using the union bound. Therefore,
\begin{align*}
P(\norm{\X - \E \X} \geq t) &\leq \sum_{\bar{\u}_1 \in C_1, \ldots, \bar{\u}_k \in C_k} P\left(\X(\bar{\u}_1,\ldots,\bar{\u}_k) - \E[\X(\bar{\u}_1,\ldots,\bar{\u}_k]) \geq \frac{t}{2}\right) \\
&\leq k_0^{\sum_{i=1}^{k} d_i} \cdot 2 \exp \left( - \frac{t^2 n^2}{2 \sigma^2 (n+1) (1 - n/N)} \right).
\end{align*}

\end{proof}

Based on Theorem~\ref{th:tensor_hoefdding_no_replacement} as well as standard concentration bounds for vectors and matrices (see e.g.~\cite{tropp2015introduction}), we prove that the required sampling conditions of Condition~\ref{cond:sampling} hold for the specific sample sizes given in the next lemma.

\begin{restatable}{lemma}{samplingconditions}
\label{lemma:sampling_conditions}
Consider the sub-sampled gradient, Hessian and third-order tensor defined in Eq.~\eqref{eq:sampled_derivatives}. Under Assumption~\ref{a:continuity}, the sampling conditions in Eqs.~\eqref{eq:sampling_g},~\eqref{eq:sampling_b} and~\eqref{eq:sampling_t} are satisfied with probability $1 - \delta,$ $\delta \in (0,1)$ for the following choice of the size of the sample sets $\S^g, \S^b$ and $\S^t$:
\begin{align*}
n_g = \tilde{\bigO} \left( \nfrac{\kappa_g^2 \epsilon^{2}}{L_f^2} + \nfrac{1}{N} \right)^{-1}, \\
n_b = \tilde{\bigO} \left( \nfrac{\kappa_b^2 \epsilon^{4/3}}{L_g^2} + \nfrac{1}{N} \right)^{-1}, \\
n_t = \tilde{\bigO} \left( \nfrac{\kappa_t^2 \epsilon^{2/3}}{L_b^2} + \nfrac{1}{N} \right)^{-1},
\end{align*}
where $\tilde{\bigO}$ hides poly-logarithmic factors and a polynomial dependency to $d$.
\end{restatable}
\begin{proof}

\textbf{Gradient}

Let $n_g := | \S^g |$ and note that by the triangle inequality as well as Lipschitz continuity of $\nabla f$ (Assumption~\ref{a:continuity}) we have
\begin{equation}
\| \g(\x) \| = \frac{1}{n_g}\sum_{i \in \S^g }\| \nabla f_i(\x) \| \leq L_g := \sigma_g
\end{equation}

We then apply Theorem~\ref{th:Matrix_Hoeffding_Serfling_Inequality} on the gradient vector and require the probability of a deviation larger or equal to $t$ to be lower than some $\delta \in (0,1]$.
\begin{align}
P\left(\norm{\g(\x)-\nabla f(\x)}>t \right)
&\leq 2d \exp \bigg( -  \frac{n_g t^2}{8 \sigma_g^2 (1+1/n_g) (1- n_g/N)} \bigg) \overset{!}{\leq} \delta
\end{align}
Taking the log on both side, we get
\begin{equation}
- \frac{n_g t^2}{8 \sigma_g^2 (1+1/n_g) (1- n_g/N)} \overset{!}{\leq}  \log \frac{\delta}{2d},
\end{equation}
which implies
\begin{align}
n_g t^2 & \overset{!}{\geq}  \log \frac{2d}{\delta} \left(8 \sigma_g^2 (1+1/n_g) (1- n_g/N) \right) \nonumber \\
&= \log \frac{2d}{\delta} \left( 8 \sigma_g^2 \left(1 + 1/n_g - n_g/N - 1/N \right) \right)
\end{align}

Since $\frac{1}{n_g} - \frac{1}{N} < 1 $, we instead require the following simpler condition,
\begin{align}
n_g t^2 &\geq \log \frac{2d}{\delta} \left( 8 \sigma_g^2 \left(2 - n_g/N \right) \right) \nonumber \\
\implies&  n_g \cdot \left( t^2 + \log \frac{2d}{\delta} \frac{8 \sigma_g^2}{N}  \right) \geq \log \frac{2d}{\delta} \left( 16 \sigma_g^2 \right) \nonumber \\
\implies& n_g \geq \frac{16 \sigma_g^2 \log \frac{2d}{\delta}}{\left( t^2 + \frac{8 \sigma_g^2}{N}  \log \frac{2d}{\delta} \right)}
\end{align}
Finally, we can simply choose $t = \kappa_g \epsilon$ in order to satisfy Eq.~\eqref{eq:sampling_g}.

\textbf{Hessian}

Again, let $n_b := | \S^b |$ and note that by the triangle inequality as well as Lipschitz continuity of $\nabla^2 f$ (Assumption~\ref{a:continuity}) we have 
\begin{equation*}
\|\Bm(\x) \|\leq \frac{1}{n_b}\sum_{i \in \S^b }\| \nabla^2 f_i(\x) \| \leq L_g := \sigma_b
\end{equation*}

Now we apply the matrix Hoeffding's inequality stated in Theorem~\ref{th:Matrix_Hoeffding_Serfling_Inequality} on the Hessian and require the probability of a deviation larger or equal to $t$ to be lower than some $\delta \in (0,1]$.
\begin{align}
P\left(\norm{\Bm(\x) - \nabla^2 f(\x)} > t \right)
&\leq 2d \exp \bigg( -  \frac{n_b t^2}{8 \sigma_b^2 (1+1/n_b) (1- n_b/N)} \bigg) \overset{!}{\leq} \delta
\end{align}
Taking the log on both side, we get
\begin{equation}
- \frac{n_b t^2}{8 \sigma_b^2 (1+1/n_b) (1- n_b/N)} \leq \log \frac{\delta}{2d},
\end{equation}
which implies
\begin{align}
n_b t^2 &\geq \log \frac{2d}{\delta} (8 \sigma_b^2 (1+1/n_b) (1- n_b/N)) \nonumber \\
&= \log \frac{2d}{\delta} \left( 8 \sigma_b^2 \left(1 + 1/n_b - n_b/N - 1/N \right) \right)
\end{align}

Since $\frac{1}{n_b} - \frac{1}{N} < 1$, we instead require the following simpler condition,
\begin{align}
n_b t^2 &\geq \log \frac{2d}{\delta} \left( 8 \sigma_b^2 \left(2 - n_b/N \right) \right) \nonumber \\
\implies&  n_b \cdot \left( t^2 + \log \frac{2d}{\delta} \frac{8 \sigma_b^2}{N}  \right) \geq \log \frac{2d}{\delta} \left( 16 \sigma_b^2 \right) \nonumber \\
\implies& n_b \geq \frac{16 \sigma_b^2 \log \frac{2d}{\delta}}{\left( t^2 + \frac{8 \sigma_b^2}{N}  \log \frac{2d}{\delta} \right)}
\end{align}
Finally, we can simply choose $t = \kappa_b \epsilon^{2/3}$ in order to satisfy Eq.~\eqref{eq:sampling_b} since $\forall \s \in \R^d$,
\begin{equation}
\| (\Bm(\x) - \nabla^2 f(\x)) \s \| \leq \| (\Bm(\x) - \nabla^2 f(\x)) \| \cdot \|\s \| \leq \kappa_b \epsilon^{2/3} \| \s \|, 
\end{equation}

\textbf{Third-order derivative}
Let $n_t := | \S^t |$ and assume that $a\leq \nabla^3 f_i (\u_1,\ldots,\u_k)\leq b$ for all $i \in \left\{1,\ldots,n\right\}$ and $(\u_1,\ldots,\u_k) \in \mathbb{R}^{d_1 \times \dots \times d_k}$. Define $\sigma_t = (b-a)$.

We apply Theorem~\ref{th:tensor_hoefdding_no_replacement} and require the probability of a deviation larger or equal to $t$ to be lower than some $\delta \in (0,1]$.
\begin{align}
P\left(\norm{\Tm(\x) - \nabla^3 f(\x)} > t \right)
&\leq k_0^{3d} \cdot 2 \exp \bigg( -  \frac{n_t^2 t^2}{2 \sigma_t^2 (n_t+1) (1- n_t/N)} \bigg) \overset{!}{\leq} \delta
\end{align}
Taking the log on both side, we get
\begin{equation}
- \frac{n_t^2 t^2}{2 \sigma_t^2 (n_t+1) (1- n_t/N)} \leq \log \frac{\delta}{2 k_0^{3d}},
\end{equation}
which implies
\begin{align}
n_t^2 t^2 &\geq \log \frac{2 k_0^{3d}}{\delta} (2 \sigma_t^2 (n_t+1) (1- n_t/N)) \nonumber \\
&= \log \frac{2 k_0^{3d}}{\delta} \left(2 \sigma_t^2 \left(n_t + 1 - n_t^2/N - n_t/N \right) \right) \nonumber \\
\implies& n_t t^2 \geq \log \frac{2 k_0^{3d}}{\delta} \left(2 \sigma_t^2 \left(1 + 1/n_t - n_t/N - 1/N \right) \right)
\end{align}
Since $\frac{1}{n_t} - \frac{1}{N} < 1$, we instead require the following simpler condition,
\begin{align}
n_t t^2 &\geq \log \frac{2 k_0^{3d}}{\delta} \left(2 \sigma_t^2 \left(2 - n_t/N \right) \right) \nonumber \\
\implies&  n_t \cdot \left( t^2 + \log \frac{2 k_0^{3d}}{\delta} \frac{ 2 \sigma_t^2}{N}  \right) \geq \log \frac{2 k_0^{3d}}{\delta} 4 \sigma_t^2 \nonumber \\
\implies&  n_t \geq \frac{4 \sigma_t^2 \log \frac{2 k_0^{3d}}{\delta}}{\left( t^2 + \frac{2 \sigma_t^2}{N}  \log \frac{2 k_0^{3d}}{\delta} \right)}.
\end{align}

Finally, we can simply choose $t = \kappa_t \epsilon^{1/3}$ in order to satisfy Eq.~\eqref{eq:sampling_t} since $\forall \s \in \R^d$,
\begin{align}
\| \Tm[\s]^2 - \nabla^3 f(\x) [\s]^2 \| &\leq \| \Tm[s] - \nabla^3 f(\x)[\s] \| \| \s \| \nonumber \\
&\leq \| \Tm - \nabla^3 f(\x) \| \| \s \|^2 \nonumber \\
&\leq \kappa_t \epsilon^{1/3} \| \s \|^2. 
\end{align}

\end{proof}

\subsection{Sampling with replacement}

We now turn our attention to sampling with replacement. First, we derive a concentration bound for a sum of i.i.d. tensors using a similar proof technique as in the previous subsection. All the proofs are very similar to the case of sampling without replacement and are therefore postponed to the supplementary material.

\begin{theorem}[Tensor Hoeffding Inequality]
\label{th:tensor_hoefdding}
Let $\X$ be a sum of $n$ i.i.d. tensors $\Y_i \in \R^{d_1 \times \dots \times d_k}$. Let $\u_1, \dots \u_k$ be such that $\| \u_i \| = 1$ and assume that for each tensor $i$, $a \leq \Y_i(\u_1,\ldots,\u_k) \leq b$. Let $\sigma := (b-a)$, then we have
\begin{equation*}
P(\norm{\X - \E \X} \geq t) \leq k_0^{(\sum_{i=1}^{k} d_i)} \cdot 2 \exp\left(-\frac{t^2}{2n \sigma^2}\right),
\end{equation*}
where $k_0 = \left(\frac{2k}{\log(3/2)}\right)$.
\end{theorem}

Based on Theorem~\ref{th:tensor_hoefdding} as well as standard concentration bounds for vectors and matrices (see e.g.~\cite{tropp2015introduction}), we prove the required sampling conditions of Condition~\ref{cond:sampling} hold for the specific sample sizes given in the next lemma.

\begin{lemma}
\label{lemma:sampling_conditions_with_replacement}
Consider the sub-sampled gradient, Hessian and third-order tensor defined in Eq.~\eqref{eq:sampled_derivatives}. The sampling conditions in Eqs.~\eqref{eq:sampling_g},~\eqref{eq:sampling_b} and~\eqref{eq:sampling_t} are satisfied with probability $1 - \delta,$ $\delta \in (0,1)$ for the following choice of the size of the sample sets $\S^g, \S^b$ and $\S^t$:
\begin{align}
n_g = \tilde{\bigO} \left( \frac{L_f^2}{\kappa_g^2 \epsilon^2} \right),
n_b = \tilde{\bigO} \left( \frac{L_g^2}{\kappa_b^2 \epsilon^{4/3}} \right),
n_t = \tilde{\bigO} \left( \frac{L_b^2}{\kappa_t^2 \epsilon^{2/3}} \right),
\end{align}
where $\tilde{\bigO}$ hides poly-logarithmic factors and a polynomial dependency to $d$.
\end{lemma}
\begin{proof}

The proof consists in using concentration inequalities to prove that there exists a sample size such that the sampled quantity is close enough to the expected value.\\

\textbf{Gradient}

Let $n_g := | \S^g |$ and note that by the triangle inequality as well as Lipschitz continuity of $\nabla f$ (Assumption~\ref{a:continuity}) we have
\begin{equation}
\| \g(\x) \| = \frac{1}{n_g}\sum_{i \in \S^g }\| \nabla f_i(\x) \| \leq L_g := \sigma_g
\end{equation}

We then apply Theorem~\ref{thm:intro-hoeffding} on the gradient vector and require the probability of a deviation larger or equal to $t$ to be lower than some $\delta \in (0,1]$.

\begin{align}
P\left(\norm{\g(\x)-\nabla f(\x)}>t \right)
&\leq d\exp \left( \frac{-t^2}{8 \sigma^2_g/n_g} \right) \overset{!}{\leq} \delta
\end{align}

Taking the log on both side, we get
\begin{equation}
\frac{-t^2}{8\sigma_g^2/n_g} = \frac{-t^2 \cdot n_g}{8 \sigma_g^2} \leq \log \frac{\delta}{d},
\end{equation}
which implies
\begin{align}
t^2 \cdot n_g &\geq (8 \sigma_g^2) \log \frac{d}{\delta} \nonumber \\
\implies& n_g \geq \frac{8 \sigma_g^2}{t^2} \log \frac{d}{\delta}.
\end{align}

Finally, we can simply choose $t = \kappa_g \epsilon$ in order to satisfy Eq.~\eqref{eq:sampling_g}.

\textbf{Hessian}

Again, let $n_b := | \S^b |$ and note that by the triangle inequality as well as Lipschitz continuity of $\nabla^2 f$ (Assumption~\ref{a:continuity}) we have 
\begin{equation*}
\|\Bm(\x) \|\leq \frac{1}{n_b}\sum_{i \in \S^b }\| \nabla^2 f_i(\x) \| \leq L_g := \sigma_b
\end{equation*}

Now we apply Theorem~\ref{thm:intro-hoeffding} on the Hessian and require the probability of a deviation larger or equal to $t$ to be lower than some $\delta \in (0,1]$.

\begin{align}
P\left(\norm{\Bm(\x) - \nabla^2 f(\x)} > t \right)
&\leq d\exp \left( \frac{-t^2}{8 \sigma_b^2/n_b} \right) \overset{!}{\leq} \delta
\end{align}

Taking the log on both side, we get
\begin{equation}
\frac{-t^2}{8\sigma_b^2/n_b} = \frac{-t^2 \cdot n_b}{8 \sigma_b^2} \leq \log \frac{\delta}{d},
\end{equation}
which implies
\begin{align}
t^2 \cdot n_b &\geq (8 \sigma_b^2) \log \frac{d}{\delta} \nonumber \\
\implies& n_b \geq \frac{8 \sigma_b^2}{t^2} \log \frac{d}{\delta}.
\end{align}

Finally, we can simply choose $t = \kappa_b \epsilon^{2/3}$ in order to satisfy Eq.~\eqref{eq:sampling_b} since $\forall \s \in \R^d$,
\begin{equation}
\| (\Bm(\x) - \nabla^2 f(\x)) \s \| \leq \| (\Bm(\x) - \nabla^2 f(\x)) \| \cdot \|\s \| \leq \kappa_b \epsilon^{2/3} \| \s \|, 
\end{equation}

\textbf{Third-order derivative}

We apply Theorem~\ref{th:tensor_hoefdding} on the normalized~\footnote{Note that we here consider the normalized sum. The reader can verify that the bound of Theorem~\ref{th:tensor_hoefdding} becomes $P(\norm{\X - \E \X} \geq t) \leq k_0^{(\sum_{i=1}^{k} d_i)} \cdot 2 \exp\left(-\frac{t^2 n}{2 \sigma^2}\right)$.} third-order derivative. Let $n_t := | \S^t |$ and define
\begin{equation}
\mathcal{Z} = \frac{1}{n_t} \sum_{i \in \S^t} \nabla^3 f_i(\x) - \nabla^3 f(\x) = \Tm(\x) - \nabla^3 f(\x).
\end{equation}

We then require the probability of a deviation larger or equal to $t$ to be lower than some $\delta \in (0,1]$.

\begin{align}
P\left(\norm{\Tm(\x) - \nabla^3 f(\x)} > t \right)
&\leq k_0^{3d} \cdot 2 \exp \left( -\frac{t^2 n_t}{2 \sigma_t^2} \right)  \overset{!}{\leq} \delta
\end{align}

Taking the log on both side, we get
\begin{equation}
-\frac{t^2 n_t}{2 \sigma_t^2} \leq \log \frac{\delta}{2 k_0^{3d}},
\end{equation}
which implies
\begin{align}
n_t \geq \frac{2 \sigma_t^2}{t^2} \log \frac{2 k_0^{3d}}{\delta}.
\end{align}

Finally, we can simply choose $t = \kappa_t \epsilon^{1/3}$ in order to satisfy Eq.~\eqref{eq:sampling_t} since $\forall \s \in \R^d$,
\begin{align}
\| \Tm[\s]^2 - \nabla^3 f(\x) [\s]^2 \| &\leq \| \Tm[s] - \nabla^3 f(\x)[\s] \| \| \s \| \nonumber \\
&\leq \| \Tm - \nabla^3 f(\x) \| \| \s \|^2 \nonumber \\
&\leq \kappa_t \epsilon^{1/3} \| \s \|^2. 
\end{align}

\end{proof}

\begin{remark}[Adaptive sampling strategy]
We note that prior work, such as~\citep{cartis2011adaptive, kohler2017sub} among others, have used sampling strategies that are iteration adaptive (i.e. the sample size changes at each iteration $k$). Such a strategy is also possible with our approach and would require a simple modification to Condition~\ref{cond:sampling}. For instance, in the case of the gradient, one could modify the condition $\| \g_k - \nabla f(\x_k) \| \leq \kappa_g \epsilon$ to $\| \g_k - \nabla f(\x_k) \| \leq \kappa_g \| \s \|$. This would change the requirement on the sample size to depend on the step length $\| \s_k \|$, which means that the sample size would be adapted at each iteration $k$. As pointed out in~\citep{kohler2017sub}, a potential drawback is that for a given iteration $k$, the step $\s_k$ is yet to be determined. Based on the Lipschitz continuity of the involved functions, ~\citep{kohler2017sub} argued that the previous step was a fair estimator of the current one, which was confirmed experimentally.
\end{remark}

\section{Worst-case complexity analysis}
\label{sec:analysis}

In the following, we provide a proof of convergence of \methodname~to a third-order critical point, i.e. a point $\x^*$ such that
\begin{equation}
    \chi_{f,i}(\x^*) \leq \epsilon_i \quad \text{ for } i=1, \dots, 3.
\label{eq:convergence_criterion}.
\end{equation}

The high-level idea of the analysis is to first show that the model decreases proportionally to the criticality measures at each iteration and then relate the model decrease to the function decrease. Since the function $f$ is lower bounded, it can only decrease a finite number of times, which therefore implies convergence. We start with a bound on the model decrease in terms of the step length $\| \s \|$.

\begin{restatable}{lemma}{modeldecrease}
\label{lemma:model_decrease}
For any $\x_k \in \R^d$, the step $\s_k$ (satisfying Condition~\ref{cond:approximate_min}) is such that
\begin{align}
\phi_k({\bf 0}) - \phi_k(\s_k) > \frac{\sigma_k}{4} \| \s_k \|^4.
\label{eq:decrease_model}
\end{align}
\end{restatable}
\begin{proof}
Note that $m_k({\bf 0}) = f(\x_k)$. Using the optimality conditions introduced in Condition~\ref{cond:approximate_min}, we get
\begin{align}
0 < m_k({\bf 0}) - m_k(\s_k) = \phi_k({\bf 0}) - \phi_k(\s_k) - \frac{\sigma_k}{4} \| \s_k \|^4,
\end{align}
which directly implies the desired result.

\end{proof}

In order to complete our claim of model decrease, we prove that the length of the step $\s_k$ can not be arbitrarily small compared to the first three derivatives of the objective function.

\subsection{Auxiliary lemmas}

\begin{restatable}{lemma}{boundsnorm}
\label{lemma:bound_s_norm}
Suppose that Condition~\ref{cond:sampling} holds with the choice $\kappa_g = \frac14$, $\kappa_b = \frac14$, $\kappa_t = \frac12$. For any $\x_k \in \R^d$, the length of the step $\s_k$ (satisfying Condition~\ref{cond:approximate_min}) is such that
\begin{equation}
\| \s_k \| \geq \kappa_k^{-1/3} \left( \chi_{f,1}(\x_k + \s_k) - \frac12 \epsilon_1 \right)^{1/3},
\label{eq:bound_s_norm}
\end{equation}
where $\kappa_k = \left( \sigma_k + \frac{L_t}{2} + \theta + \frac{1}{4} \right)$.
\end{restatable}

\begin{proof}

We start with the following bound,
\begin{align}
\| \nabla f(\x_k + \s_k) \| \leq \| \nabla f(\x_k + \s_k) - \nabla \phi_k(\s_k) \| + \| \nabla \phi_k(\s_k) \|.
\label{eq:bound_nabla_f}
\end{align}

We bound the second term in the RHS of Eq.~\eqref{eq:bound_nabla_f} using the termination condition presented in Eq.~\eqref{eq:termination_criterion}. We get
\begin{align}
\| \nabla \phi_k(\s_k) \| &\leq \left \| \nabla \phi_k(\s_k) + \sigma_k \s_k \| \s_k \|^2 \right \| + \sigma_k \| \s_k \|^3 \nonumber \\
&= \| \nabla m_k(\s_k) \| + \sigma_k \| \s_k \|^3 \nonumber \\
&\leq \theta \| \s_k \|^3 + \sigma_k \| \s_k \|^3 \nonumber \\
&\leq \left( \theta + \sigma_k \right) \| \s_k \|^3.
\end{align}

For the first term in the RHS of Eq.~\eqref{eq:bound_nabla_f}, we will use the following standard inequality (see e.g.~\citep{cartis2020concise,cartis2022evaluation}):
\begin{equation}
\| \nabla f(\y) - \nabla f(\x) - \nabla^2 f(\x) (\y-\x) - \nabla^3 f(\x)[\y-\x]^2 \| \leq \frac{L_t}{2} \| \y-\x \|^3 \quad \forall \x, \y \in \R^d.
\label{eq:third_order_Lipschitz}
\end{equation}

This gives us the following bound
\begin{align}
\| \nabla f(\x_k + & \s_k) - \nabla \phi_k(\s_k) \| = \left \| \nabla f(\x_k + \s_k) - \g_k - \Bm_k \s_k - \frac{1}{2} \Tm_k [\s_k]^2  \right \| \nonumber \\
&\leq \left \| \nabla f(\x_k + \s_k) - \nabla f(\x_k) - \nabla^2 f(\x_k) \s_k - \frac{1}{2} \nabla^3 f(\x_k) [\s_k]^2  \right \| \nonumber \\
&+ \| \g_k - \nabla f(\x_k)\| + \| (\Bm_k - \nabla^2 f(\x_k)) \s_k \| + \frac{1}{2} \| \Tm_k[\s_k]^2 - \nabla^3 f(\x_k) [\s_k]^2 \| \nonumber \\
&\leq \frac{L_t}{2}  \| \s_k \|^3 + \kappa_g \epsilon + \kappa_b \epsilon^{2/3} \| \s_k \| +\frac{1}{2} \kappa_t \epsilon^{1/3} \| \s_k \|^2,
\end{align}
where the last inequality uses Eq.~\eqref{eq:third_order_Lipschitz} and Condition~\ref{cond:sampling} where we set $\epsilon_1 = \epsilon$.

Next, we apply the Young's inequality for products which states that if $a, b \in \R_{\geq 0}$ and $p, q \in \R_{>1}$ such that $1/p + 1/q = 1$ then
\begin{equation}
ab \le \frac{a^p}{p} + \frac{b^q}{q}.
\label{eq:young_ineq}
\end{equation}

We obtain
\begin{align}
\| \nabla f(\x_k + \s_k) \| &\leq \left( \theta + \sigma_k \right) \| \s_k \|^3 + \frac{L_t}{2} \| \s_k \|^3 + \kappa_g \epsilon + \frac{ \kappa_b}{3} \left( 2 \epsilon +  \| \s_k \|^3 \right) + \frac{\kappa_t}{6}  \left( \epsilon + 2 \| \s_k \|^3 \right) \nonumber \\
&= \left( \sigma_k + \theta + \frac{L_t}{2} + \frac{\kappa_b}{3} + \frac{\kappa_t}{3} \right) \| \s_k \|^3 + \left( \kappa_g + \frac{2 \kappa_b}{3} + \frac{\kappa_t}{6} \right) \epsilon,
\end{align}

therefore
\begin{equation}
\left( \sigma_k + \theta + \frac{L_t}{2} + \frac{\kappa_b}{3} + \frac{\kappa_t}{3} \right)^{-1} \left( \| \nabla f(\x_k + \s_k) \| - \left( \kappa_g + \frac{2 \kappa_b}{3} + \frac{\kappa_t}{6} \right) \epsilon \right) \leq \| \s_k \|^3.
\end{equation}

Choosing $\kappa_g = \frac{1}{4}$, $\kappa_b = \frac{1}{4}$, $\kappa_t = \frac12$,
\begin{equation}
\left( \sigma_k + \theta + \frac{L_t}{2} + \frac{1}{4} \right)^{-1} \left( \| \nabla f(\x_k + \s_k) \| - \frac12 \epsilon \right) \leq \| \s_k \|^3.
\end{equation}

\end{proof}

We next prove that the length of the step $\s_k$ can not be arbitrarily small compared to $\chi_{f,2}$. We first need an additional lemma that relates the step length $\s_k$ to the second criticality measure $\chi_{f,2}$. Proving such result requires the following auxiliary lemma.

\begin{restatable}{lemma}{BLipschitz}
Suppose that Condition~\ref{cond:sampling} holds. For all $\x_k, \s \in \R^d$,
\begin{align}
\| \nabla^2 f(\x_k + \s) - \nabla_\s^2 \phi_k(\s) \|
\leq \left( \frac{L_t}{2} + \frac{\kappa_t}{2} \right) \| \s \|^2 + \left(\kappa_b + \frac{\kappa_t}{2} \right) \epsilon_2.
\label{eq:B_Lipschitz}
\end{align}
\label{lemma:B_Lipschitz}
\end{restatable}
\begin{proof}

We will use the following standard inequality (see e.g.~\citep{cartis2020concise,cartis2022evaluation}):
\begin{equation}
\| \nabla^2 f(\y) - \nabla^2 f(\x) - \nabla^3 f(\x) (\y-\x) \| \leq \frac{L_t}{2} \sqnorm{\y-\x} \quad \forall \x, \y \in \R^d.
\label{eq:third_order_Lipschitz_2}
\end{equation}

Recall that
\begin{equation}
\nabla^2_\s \phi_k(\s) = \Bm_k + \Tm_k [\s],
\end{equation}
therefore
\begin{align}
& \| \nabla^2 f(\x_k + \s) - \nabla_\s^2 \phi_k(\s) \| 
= \| \nabla^2 f(\x_k + \s) - \Bm_k - \Tm_k [\s] \| \nonumber \\
&\quad \leq \| \nabla^2 f(\x_k + \s) - \nabla^2 f(\x_k) - \nabla^3 f(\x_k)[\s] \| + \| \nabla^2 f(\x_k) - \Bm_k \| + \| \nabla^3 f(\x_k)[\s] - \Tm_k [\s] \| \nonumber \\
&\quad \leq \frac{L_t}{2} \| \s \|^2 + \kappa_b \epsilon_2 + \kappa_t \epsilon_2^{1/2} \| \s \|, \nonumber
\label{eq:hessian_f_m_delta}
\end{align}
where the last inequality uses Eq.~\eqref{eq:third_order_Lipschitz_2} and Condition~\ref{cond:sampling} with $\epsilon_2 = \epsilon^{2/3}$.

We again apply the Young's inequality for products stated in Eq.~\ref{eq:young_ineq} which yields
\begin{align}
\| \nabla^2 f(\x_k + \s) - \nabla_\s^2 \phi_k(\s) \|
\leq \frac{L_t}{2} \| \s \|^2 + \kappa_b \epsilon_2 + \frac{\kappa_t}{2} \| \s \|^2 + \frac{\kappa_t}{2} \epsilon_2.
\end{align}

\end{proof}


\begin{restatable}{lemma}{boundsnormxitwo}
\label{lemma:bound_s_norm_xi2}
Suppose that Condition~\ref{cond:sampling} holds with the choice $\kappa_g = \frac14$, $\kappa_b = \frac14$, $\kappa_t = \frac12$. For any $\x_k \in \R^d$, the length of the step $\s_k$ (satisfying Condition~\ref{cond:approximate_min}) is such that
\begin{align}
\|\s_k\| \geq \kappa_{k,2}^{-1/2} \left( \chi_{f,2}(\x_k + \s_k) - \frac{1}{2} \epsilon_2 \right)^{1/2},
\label{eq:lower_stepbound_lambda}
\end{align}
where $\kappa_{k,2} = \left( 3\sigma_k + \frac{L_t}{2} + \theta + \frac14 \right)$.
\end{restatable}
\begin{proof}

Using the definition of the model in Eq.~\eqref{eq:model} and the fact that $$\min_\z[a(\z)+b(\z)] \geq \min_\z[a(\z)] + \min_\z[b(\z)],$$
we find that
\begin{align}
\lambda_{min}(\nabla^2 f(\x_k + & \s_k)) = \min_{\|\y\|=1} \nabla^2 f(\x_k+\s_k)[\y]^2  \\
&= \min_{\|\y\|=1} \left( \nabla^2 f(\x_k+\s_k) - \nabla_\s^2 \phi_k(\s_k) - \frac{\sigma_k}{4} \nabla_\s^2 \|\s_k\|^{4} + \nabla_\s^2 m_k(\s_k) \right) [\y]^2 \nonumber \\
&\geq \min_{\|\y\|=1} \left( \nabla^2 f(\x_k+\s_k) - \nabla_\s^2 \phi_k(\s_k) \right) [\y]^2 + \frac{\sigma_k}{4} \min_{\|\y\|=1} \left( - \nabla_\s^2 \|\s_k\|^{4} \right) [\y]^2
+ \min_{\|\y\|=1} \nabla_\s^2 m_k(\s_k) [\y]^2 \nonumber
\end{align}

Considering each term in turn, and using Lemma~\ref{lemma:B_Lipschitz}, we see that
\begin{align}
\min_{\|\y\|=1} & \left( \nabla^2 f(\x_k+\s_k) - \nabla_\s^2 \phi_k(\s_k) \right) [\y]^2 \nonumber \\
& \geq \min_{\|\y_1\|=\|\y_2\|=1} \left( \nabla^2 f(\x_k+\s_k) - \nabla_\s^2 \phi_k(\s_k) \right) [\y_1,\y_2] \nonumber \\
& \geq - \max_{\|\y_1\|=\|\y_2\|=1} \left | \left( \nabla^2 f(\x_k+\s_k) - \nabla_\s^2 \phi_k(\s_k) \right) [\y_1,\y_2] \right | \nonumber \\
& = - \| \nabla^2 f(\x_k+\s_k)  - \nabla_s^2 \phi_k(\s_k) \|_{[2]}^{} \nonumber \\
&\stackrel{\eqref{eq:B_Lipschitz}}{\geq} - \left( \frac{L_t}{2} + \frac{\kappa_t}{2} \right) \| \s_k \|^2 - \left(\kappa_b + \frac{\kappa_t}{2} \right) \epsilon_2.
\end{align}

Note that for the second-order derivative $\nabla_\s^2 m_k(\s)$, we get:
\begin{equation}
\nabla^2_\s m_k(\s) = \Bm_k + \Tm_k [\s] + \frac{\sigma_k}{4} \nabla^2_\s \| \s \|^4,
\end{equation}
where
\begin{equation}
\frac14 \nabla^2_\s \| \s \|^4 = \nabla_\s \; \s \| \s \|^2 = \| \s \|^2 \Im + 2 \s \s^\top \succcurlyeq \| \s \|^2   
\label{eq:hessian_s}
\end{equation}

Using Eq.~\eqref{eq:hessian_s}, we get that $\frac14 \nabla_\s^2 \left(\|\s_k\|^{4} \right) [\y]^2= 2 (\s_k^\top \y)^2 + \| \s_k \|^2 \| \y \|^2$, therefore
\begin{align}
\min_{\|\y\|=1} &\left( - \nabla_\s^2 (\|\s_k\|^4) \right) [\y]^2 
= - \max_{\|\y\|=1} \nabla_\s^2 (\|\s_k\|^4) [\y]^2 
= - 12 \|\s_k\|^2.
\end{align}

From Eq.~\eqref{eq:def_chi_m}, we have $\min_{\|\y\|=1} \nabla_s^2 m_k(\s_k) [\y]^2 = \lambda_{min}(\nabla_\s^2 m_k(\s_k))$. Combined with the last two equations, we get that
\begin{align}
-\lambda_{min}(\nabla^2 f(\x_k + \s_k)) &\leq \left( \frac{L_t}{2} + \frac{\kappa_t}{2} \right) \| \s_k \|^2 \nonumber \\
& + \left(\kappa_b + \frac{\kappa_t}{2} \right) \epsilon_2 + 3\sigma_k \|\s_k\|^2 -\min[0, \lambda_{min}(\nabla_\s^2 m_k(\s_k))].
\label{eq:lambda_f}
\end{align}

As the right hand side of the above equation is non-negative, we can rewrite Eq.~\eqref{eq:lambda_f} as
\begin{align}
\max[0, -\lambda_{min}(\nabla^2 f(\x_k+\s_k))] &\leq  \left( \frac{L_t}{2} + \frac{\kappa_t}{2} +  3\sigma_k \right) \|\s_k\|^2 + \left(\kappa_b + \frac{\kappa_t}{2} \right) \epsilon_2 \nonumber \\
&+ \max[0, -\lambda_{min}(\nabla_\s^2 m_k(\s_k))].
\end{align}

Combining the above with Eq.~\eqref{eq:def_chi_f} and Eq.~\eqref{eq:def_chi_m}, and with Eq.~\eqref{eq:termination_criterion} for $i=2$, we conclude
\begin{align}
\chi_{f,2}(\x_k + \s_k) &\leq  \left( \frac{L_t}{2} + \frac{\kappa_t}{2} +  3\sigma_k \right) \|\s_k\|^2 + \left(\kappa_b + \frac{\kappa_t}{2} \right) \epsilon_2 + \chi_{m,2}(\x_k,\s_k) \nonumber \\
&\leq \left( \frac{L_t}{2} + \frac{\kappa_t}{2} +  3\sigma_k + \theta \right) \|\s_k\|^2 + \left(\kappa_b + \frac{\kappa_t}{2} \right) \epsilon_2,
\end{align}
which implies
\begin{align}
\left( \frac{L_t}{2} + \frac{\kappa_t}{2} + 3\sigma_k + \theta \right) \|\s_k\|^2 \geq \chi_{f,2}(\x_k + \s_k) -  \left(\kappa_b + \frac{\kappa_t}{2} \right) \epsilon_2.
\end{align}

Choosing $\kappa_b = \frac{1}{4}, \kappa_t = \frac{1}{2}$, we conclude
\begin{align}
\|\s_k\|^2 \geq \left( 3\sigma_k + \frac{L_t}{2} + \theta + \frac14 \right)^{-1} \left( \chi_{f,2}(\x_k + \s_k) - \frac12 \epsilon_2 \right) .
\end{align}

\end{proof}


Finally, we prove that the length of the step $\s_k$ can again not be arbitrarily small compared to $\chi_{f,3}$.

\begin{restatable}{lemma}{TLipschitz}
Suppose that Condition~\ref{cond:sampling} holds. For all $\x_k, \s \in \R^d$,
\begin{align}
\| \nabla^3 f(\x_k + \s) - \nabla_\s^3 \phi_k(\s) \|
\leq L_t \| \s \| + \kappa_t \epsilon_3.
\label{eq:T_Lipschitz}
\end{align}
\label{lemma:T_Lipschitz}
\end{restatable}
\begin{proof}

Recall that
\begin{equation}
\nabla^3_\s \phi_k(\s) = \Tm_k,
\end{equation}
therefore
\begin{align}\label{eq:T_f_m_delta}
\| \nabla^3 f(\x_k + \s) - \nabla_\s^3 \phi_k(\s) \|
&= \| \nabla^3 f(\x_k + \s) - \Tm_k \| \nonumber \\
&\leq \| \nabla^3 f(\x_k + \s) - \nabla^3 f(\x_k) \| + \| \nabla^3 f(\x_k) - \Tm_k \| \nonumber \\
&\leq L_t \| \s \| + \kappa_t \epsilon_3,
\end{align}
where the last inequality uses the Lipschitz property of $\nabla^3 f(\cdot)$ and Condition~\ref{cond:sampling} with $\epsilon_3 = \epsilon^{1/3}$.

\end{proof}

\begin{restatable}{lemma}{boundsnormxithree}
\label{lemma:bound_s_norm_xi3}
Suppose that Condition~\ref{cond:sampling} holds with the choice $\kappa_g = \frac14$, $\kappa_b = \frac14$, $\kappa_t = \frac12$. For any $\x_k \in \R^d$, the length of the step $\s_k$ (satisfying Condition~\ref{cond:approximate_min}) is such that
\begin{align}
\|\s_k\| \geq \kappa_{k,3}^{-1} \left( \chi_{f,3}(\x_k + \s_k) - \frac12 \epsilon_3 \right),
\label{eq:lower_stepbound_lambda3}
\end{align}
where $\kappa_{k,3} = \left( L_t + \frac{\sigma_k}{2} + \theta \right)$.
\end{restatable}
\begin{proof}

Using the definition of the model in Eq.~\eqref{eq:model} and the fact that $$\max_\z[a(\z)+b(\z)] \leq \max_\z[a(\z)] + \max_\z[b(\z)],$$
we find that
\begin{align}
\max_{\y \in \M_{k+1}} \nabla^3 & f(\x_k+\s_k)[\y]^3
= \max_{\y \in \M_{k+1}} \left( \nabla^3 f(\x_k+\s_k) \pm \nabla_\s^3 m_k(\s_k) \right) [\y]^2 \nonumber \\
&= \max_{\y \in \M_{k+1}} \left( \nabla^3 f(\x_k+\s_k) - \nabla_\s^3 \phi_k(\s_k) - \frac{\sigma_k}{4} \nabla_\s^3 \|\s_k\|^{4} + \nabla_\s^3 m_k(\s_k) \right) [\y]^3 \nonumber \\
&\leq \max_{\y \in \M_{k+1}} \left( \nabla^3 f(\x_k+\s_k) - \nabla_\s^3 \phi_k(\s_k) \right) [\y]^3 + \frac{\sigma_k}{4} \max_{\y \in \M_{k+1}} \left( - \nabla_\s^3 \|\s_k\|^{4} \right) [\y]^3 + \max_{\y \in \M_{k+1}} \nabla_\s^3 m_k(\s_k) [\y]^3
\end{align}

Considering each term in turn, and using Lemma~\ref{lemma:T_Lipschitz}, we see that
\begin{align}
\max_{\y \in \M_{k+1}} & \left( \nabla^3 f(\x_k+\s_k) - \nabla_\s^3 \phi_k(\s_k) \right) [\y]^3 \nonumber \\
& \leq \max_{\|\y_1\|=\|\y_2\|=\|\y_3\|=1} \left( \nabla^3 f(\x_k+\s_k) - \nabla_\s^3 \phi_k(\s_k) \right) [\y_1,\y_2,\y_3] \nonumber \\
& = \| \nabla^3 f(\x_k+\s_k)  - \nabla_s^3 \phi_k(\s_k) \|_{[3]} \nonumber \\
&\stackrel{\eqref{eq:T_Lipschitz}}{\leq} L_t \| \s_k \| + \kappa_t \epsilon_3.
\end{align}

One can also show~\footnote{Note that given a vector $\y \in \R^d$ and a third-order tensor $T \in \R^{d \times d \times d}$, we have $T[\y]^3 = \sum_{i,j,k} T_{i,j,k} y_i y_j y_k$ where $T_{i,j,k}$ denotes the $(i,j,k)$-th element of $T$.} that
\begin{equation}
\max_{\y \in \M_{k+1}} \left( - \nabla_\s^3 \|\s_k\|^{4} \right) [\y]^3 = - \max_{\y \in \M_{k+1}} 2 \s_k^\top \y \| \y \|^2 \leq 2 \| \s_k \|
\end{equation}

Combining the last two equations, we get that
\begin{align}
\chi_{f,3}(\x_k + \s_k) = \max_{\y \in \M_{k+1}} \nabla^3 f(\x_k+\s_k)[\y]^3 &\leq \left( L_t + \frac{\sigma_k}{2} \right) \| \s_k \| + \kappa_t \epsilon_3 + \| \nabla_\s^3 m_k(\s_k) \| \nonumber \\
&\stackrel{\eqref{eq:termination_criterion}}{\leq} \left( L_t + \frac{\sigma_k}{2} + \theta \right) \| \s_k \| + \kappa_t \epsilon_3
\label{eq:lambda_f_3}
\end{align}
which implies
\begin{align}
\left( L_t + \frac{\sigma_k}{2} + \theta \right) \| \s_k \| \geq \chi_{f,3}(\x_k + \s_k) - \kappa_t \epsilon_3
\end{align}

Choosing $\kappa_t = \frac{1}{2}$, we conclude
\begin{align}
\|\s_k\| \geq \left( L_t + \frac{\sigma_k}{2} + \theta \right)^{-1} \left( \chi_{f,3}(\x_k + \s_k) - \frac12 \epsilon_3 \right) .
\end{align}

\end{proof}


A key lemma to derive a worst-case complexity bound on the total number of iterations required to reach a third-order critical point is to bound the number of unsuccessful iterations $|\mathcal{U}_k|$ as a function of the number of successful ones $|\mathcal{S}_k|$, that have occurred up to some iteration $k > 0$.

\begin{lemma}
\label{lemma:total_nb_steps}
The steps produced by Algorithm~\ref{alg:stm} guarantee that if $\sigma_k \leq \sigma_{max}$ for $\sigma_{max} > 0$, then the total number of iterations $k = |\U_k| + |\S_k|$ is such that $k \leq C(\gamma_1, \gamma_2, \sigma_{max}, \sigma_0)$ where
\begin{equation*}
C(\gamma_1, \gamma_2, \sigma_{max}, \sigma_0) := \left( 1 + \frac{|\log \gamma_1|}{\log \gamma_2} \right)|\S_k| + \frac{1}{\log \gamma_2} \log \left( \frac{\sigma_{max}}{\sigma_0} \right).
\end{equation*}
\end{lemma}

The proof of this Lemma can be found in~\cite{cartis2011adaptive} (Theorem 2.1). A closed-form expression for $\sigma_{max}$ is provided in Lemma~\ref{lemma:successful_sigma} in the supplementary material.

\subsection{Main result}

We are now ready to state the main result of this section that provides a bound on the number of iterations required to reach a third-order critical point.

\begin{restatable}{theorem}{worstcase}[Worst-case complexity]
\label{th:worst_case_complexity}
Let $f_{low}$ be a lower bound on $f$ and assume Condition~\ref{cond:approximate_min} holds. We define $\kappa_s = \left( \sigma_{max} + \frac{L_t}{2} + \theta + \frac14 \right), \kappa_{s,2} = \left( 3 \sigma_{max} + \frac{L_t}{2} + \theta + \frac14 \right)$, $\kappa_{s,3} = \left( L_t + \frac{\sigma_{max}}{2} + \theta \right)$ and $\kappa_{max} = \max(\sqrt[3]{2} \kappa_s^{4/3}, 2 \kappa_{s,2}^{2}, 8 \kappa_{s,3})$.
Then, given $\epsilon_i > 0, i=1 \dots 3$, with probability $1 - \delta'$ for $\delta' > 0$, Algorithm~\ref{alg:stm} needs at most
\begin{equation*}
\ceil*{\K_{succ}(\epsilon) := \frac{8 \kappa_{max} (f(\x_0) - f_{low})}{\eta_1 \sigma_{min}}  \max(\epsilon_1^{-4/3}, \epsilon_2^{-2}, \epsilon_3^{-4})}
\end{equation*}
successful iterations
and
\begin{align}
\K(\epsilon) := \ceil*{ C(\gamma_1, \gamma_2, \sigma_{max}, \sigma_0) \cdot \K_{succ}(\epsilon)}.
\label{eq:K_outer}
\end{align}
total iterations to reach an iterate $\x^*$ such that
$\chi_{f,i}(\x^*) \leq \epsilon_i \quad \text{ for } i=1,\dots,3$.
\end{restatable}
\begin{proof}
First, let $\kappa_{s} = \left( \sigma_{max} + \frac{L_t}{2} + \theta + \frac14 \right)$.
For each successful iteration $k$, the function decrease in terms of the first-order criticality measure is
\begin{align}
f(\x_k) - f(\x_{k+1}) &\geq \eta_1(f(\x_k) - \phi_k(\s_k)) \nonumber \\
&\stackrel{\eqref{eq:decrease_model}}{\geq} \frac{1}{4} \eta_1 \sigma_{min} \vectornorm{\s_k}^4_2 \nonumber \\
&\stackrel{\eqref{eq:bound_s_norm}}{\geq} \frac{1}{4} \eta_1 \sigma_{min} \kappa_k^{-4/3} \left( \| \nabla f(\x_k + \s_k) \| - \frac12 \epsilon_1 \right)^{4/3} \nonumber \\
&\geq \frac{1}{4} \eta_1 \sigma_{min} \kappa_s^{-4/3} \left(\frac12 \epsilon_1 \right)^{4/3} \nonumber \\
&\geq \frac{1}{8 \sqrt[3]{2}} \eta_1 \sigma_{min} \kappa_s^{-4/3} \epsilon_1^{4/3}
\label{eq:worst_case_complexity_f_1}
\end{align}
where the fourth inequality uses the fact that $\| \nabla f(\x_k + \s_k) \| \geq \epsilon_1$ before termination.\\

Let's now consider the function decrease in terms of the second-order criticality measure. First, let $\kappa_{s,2} = \left( 3 \sigma_{max} + \frac{L_t}{2} + \theta + \frac14 \right)$. For each successful iteration $k$, we have
\begin{align}
f(\x_k) - f(\x_{k+1}) &\geq \eta_1(f(\x_k) - \phi_k(\s_k)) \nonumber \\
&\stackrel{\eqref{eq:decrease_model}}{\geq} \frac{1}{4} \eta_1 \sigma_{min} \vectornorm{\s_k}^4_2 \nonumber \\
&\stackrel{\eqref{eq:lower_stepbound_lambda}}{\geq} \frac{1}{4} \eta_1 \sigma_{min}
\kappa_{k,2}^{-2} \left( \chi_{f,2}(\x_k + \s_k) - \frac12 \epsilon_2 \right)^{2} \nonumber \\
&\geq \frac{1}{4} \eta_1 \sigma_{min} \kappa_{k,2}^{-2} \left( \epsilon_2 - \frac12 \epsilon_2 \right)^{2} \nonumber \\
&\geq \frac{1}{4^2} \eta_1 \sigma_{min} \kappa_{s,2}^{-2} \epsilon_2^2
\label{eq:worst_case_complexity_f_2}
\end{align}
where the fourth inequality uses the fact that $\chi_{f,2}(\x_k + \s_k) \geq \epsilon_2$ before termination.\\

Lastly, we consider the function decrease in terms of the third-order criticality measure. First, let $\kappa_{s,3} = \left( L_t + \frac{\sigma_{max}}{2} + \theta \right)$. For each successful iteration $k$, we have
\begin{align}
f(\x_k) - f(\x_{k+1}) &\geq \eta_1(f(\x_k) - \phi_k(\s_k)) \nonumber \\
&\stackrel{\eqref{eq:decrease_model}}{\geq} \frac{1}{4} \eta_1 \sigma_{min} \vectornorm{\s_k}^4_2 \nonumber \\
&\stackrel{\eqref{eq:lower_stepbound_lambda3}}{\geq} \frac{1}{4} \eta_1 \sigma_{min}
\kappa_{k,3}^{-1} \left( \chi_{f,3}(\x_k + \s_k) - \frac12 \epsilon_3 \right)^{4} \nonumber \\
&\geq \frac{1}{4} \eta_1 \sigma_{min} \kappa_{k,3}^{-1} \left( \epsilon_3 - \frac12 \epsilon_3 \right)^{4} \nonumber \\
&\geq \frac{1}{4^3} \eta_1 \sigma_{min} \kappa_{s,3}^{-1} \epsilon_3^4
\label{eq:worst_case_complexity_f_3}
\end{align}
where the fourth inequality uses the fact that $\chi_{f,3}(\x_k + \s_k) \geq \epsilon_3$ before termination.\\

Thus on any successful iteration until termination we can guarantee the minimal of the decreases in Eqs.~\eqref{eq:worst_case_complexity_f_1}, ~\eqref{eq:worst_case_complexity_f_2} and~\eqref{eq:worst_case_complexity_f_3}, and hence,
\begin{equation}
f(\x_0) - f(\x_{k+1}) \geq \frac{1}{8} \eta_1 \sigma_{min} \min \left(\frac{\kappa_s^{-4/3}}{\sqrt[3]{2}}, \frac{\kappa_{s,2}^{-2}}{2}, \frac{\kappa_{s,3}^{-1}}{8} \right) \min(\epsilon_1^{4/3}, \epsilon_2^2, \epsilon_3^4) | \S_k |
\end{equation}

Using that $f$ is bounded below by $f_{low}$, we  conclude
\begin{equation}
| \S_k | \leq \frac{8 \max(\sqrt[3]{2} \kappa_s^{4/3}, 2 \kappa_{s,2}^{2}, 8 \kappa_{s,3}) (f(\x_0) - f_{low})}{\eta_1 \sigma_{min}}  \max(\epsilon_1^{-4/3}, \epsilon_2^{-2}, \epsilon_3^{-4}).
\end{equation}

We can then use Lemma~\ref{lemma:total_nb_steps} to get a bound on the total number of iterations.

Finally, we recall that our goal is to show that the result of the theorem holds with probability $1 - \delta'$. We note that the proof relies on Eq.~\eqref{lemma:bound_s_norm}, \eqref{eq:lower_stepbound_lambda} and~\eqref{eq:lower_stepbound_lambda3} which are shown to hold for each iteration with probability greater than $1 - \delta$. Then the event that the concentration conditions hold for all $T$ iterations of the algorithm has probability greater than $1 - \delta T \geq 1 - \delta'$, i.e. we need $\delta' = \frac{\delta}{T}$. 

\end{proof}

As in~\cite{cartis2020concise}, we obtain a worst-complexity bound of the order $\bigO(\max(\epsilon_1^{-\frac{4}{3}}, \epsilon_2^{-2}))$, except that we do not require the exact computation of the function derivatives but instead rely on approximate sampled quantities. By using third-order derivatives, \methodname~also obtains a faster rate than the one achieved by a sampled variant of cubic regularization~\citep{kohler2017sub} (at most $\bigO(\epsilon^{-3/2})$ iterations for $\|\nabla f(\x^*) \| \leq \epsilon$ and $\bigO(\epsilon^{-3})$ to reach approximate nonnegative curvature).

We note that the worst-case complexity bound stated in Theorem~\ref{th:worst_case_complexity} holds with high probability, which is due to the sampling conditions in Eqs.~\eqref{eq:sampling_g},~\eqref{eq:sampling_b} and~\eqref{eq:sampling_t} that are shown to be satisfied with high probability in Lemma~\ref{lemma:sampling_conditions_with_replacement}. Alternatively, one could derive a bound on the expected number of steps required to reach a third-order critical point (e.g. following the analysis of~\citep{blanchet2016convergence}), or other convergence results such as almost sure convergence (as a consequence of the Borel-Cantelli lemma).

\paragraph{Lower bound}
A recent work by~\cite{arjevani2020second} has derived lower bounds for stochastic higher-order methods under higher-order smoothness conditions. Their results show that there are worst-case functions for which stochastic high-order methods can not match the rate of deterministic methods. It is important to point out that our result does not contradict their lower bound. For hard problem instances, the sample complexity required by our algorithm will be full batch. This is in fact not surprising. For instance, even SGD without variance reduction does not achieve the rate of the deterministic version in the worst-case. However,~\cite{friedlander2012hybrid} showed that an adaptive sampling strategy does allow SGD to achieve the deterministic rate under some conditions on the approximation error. We see our sampling condition as an analog to these first-order strategies for higher-order derivatives.
One aspect that we think is worth pointing out is that there are many practical problems where the sample size required will not be full batch and the approach will therefore make significant computational savings.

\paragraph{Condition~\ref{cond:approximate_min} required in Theorem.~\ref{th:worst_case_complexity}}
Instead of exactly solving the subproblem defined in Eq.~\eqref{eq:subproblem}, Condition~\ref{cond:approximate_min} requires an approximate solution. For first and second-order convergence guarantees, one can use the (first-order) accelerated solvers for non-convex functions proposed by~\cite{carmon2017convex,carmon2018accelerated}.
Although these solvers were not \emph{specifically} design to solve Eq.~\eqref{eq:subproblem}, they do fulfill the guarantees required in Condition~\ref{cond:approximate_min}. For convex problems, ~\cite{nesterov2015implementable} provides theoretical guarantees for a high-order regularized method under some convexity assumptions for the model. However, this problem is still unsolved for $p > 2$ in the non-convex case.


\section{Experimental results}
\label{sec:practical_implementation}
\vspace{-2mm}

\begin{figure*}
\label{fig:exp_results}
	\begin{center}
          \begin{tabular}{@{}c@{\hspace{5mm}}c@{\hspace{5mm}}c@{\hspace{5mm}}}
            \includegraphics[width=0.3\linewidth]{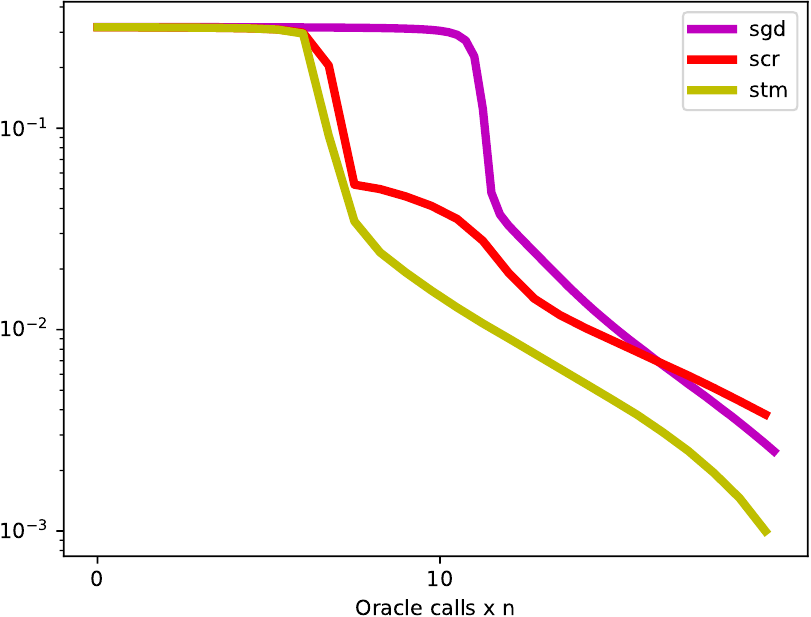} &
            \includegraphics[width=0.3\linewidth]{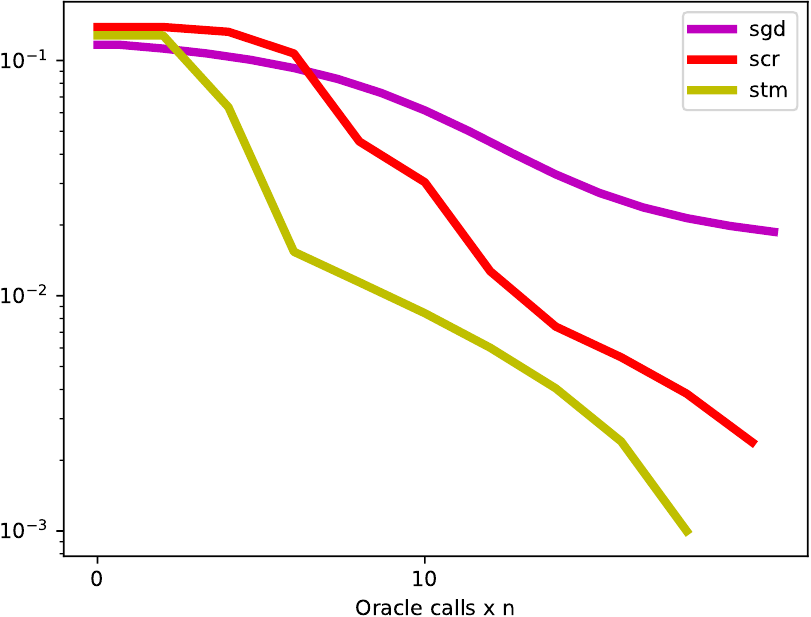} &
            \includegraphics[width=0.3\linewidth]{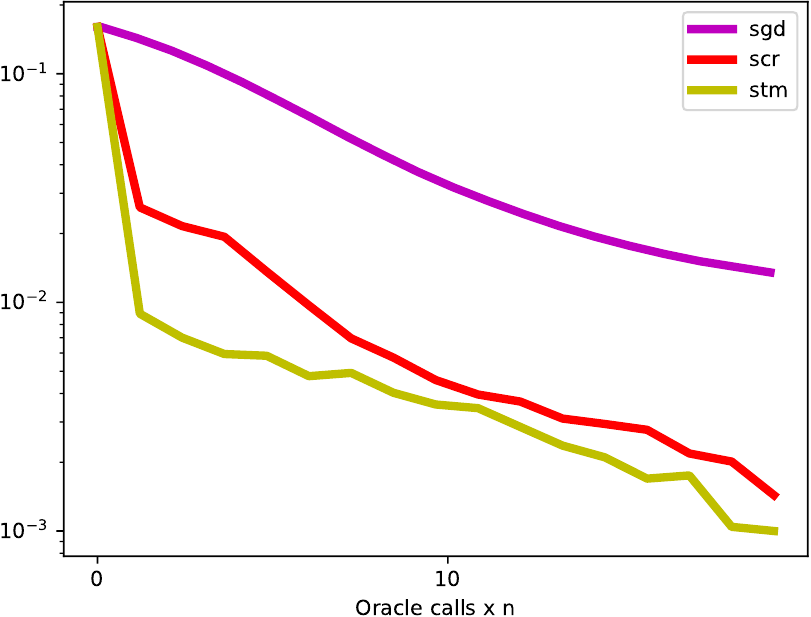}  \\
         	\scriptsize{{ A9A (n=32561, d=123)}} &
            \scriptsize{{ COVTYPE (n=581012, d=54)}} &          
            \scriptsize{{ SVMGUIDE (n=1243, d=21)}} \\
	  \end{tabular}
          \caption{\footnotesize{Log suboptimality as a function of the number of epochs. We count all oracle evaluations (including high-order derivatives) on the x-axis. Each curve is the average of 10 runs initialized from different random points.}}
          \label{fig:results}
	\end{center}
\end{figure*}

In this section, we test the performance of~\methodname~ on a non-convex logistic
regression problem similar to the ones used in~\cite{ghadimi2019generalized, zhu2020adaptive}.
Specifically, given a set of $n$ labeled datapoints $(\xi_i, y_i)_{i=1}^n$ where $\xi_i \in \R^d$ and $y_i = (0, 1)$, we consider the objective function
\begin{equation}
f(\x) = \frac12 \sum_{i=1}^n \left( \frac{1}{1 + e^{-\x^\top \xi_i}} - y_i \right)^2 + \frac{\lambda}{2} \| \x \|^2.
\end{equation}

\paragraph{Datasets}
The real-world datasets we use represent very common instances of Machine Learning problems and are part of the libsvm library~\citep{chang2011libsvm}. A summary of their main characteristic can be found in Table~\ref{table:datasets}.

\begin{table}[H]
\centering
\begin{tabular}{l|lll}
dataset & type                           & n          & d \\ \hline
 a9a                 & Classification              & $32,561$      & $123$  \\
 covtype             & Classification              & $581,012$    & $54$ \\
 svmguide3               & Classification              & $1243$ & $21$ \\
\end{tabular}
\label{table:datasets}
\caption{Overview of the real-world datasets used in our experiments.}
\end{table}

\paragraph{Practical implementation of \methodname}
We implement \methodname~as stated in Algorithm~\ref{alg:stm} using (stochastic) gradient descent steps to optimize $m_k(\s_k)$. Following~\cite{erdogdu2015convergence, kohler2017sub}, we require the sampling conditions (see Section~\ref{sec:sampling_conditions}) to hold with probability $\bigO(1-1/d)$.

We here provide additional results and briefly describe the baseline algorithms used in the experiments as well as the choice of hyper-parameters. All experiments were run on a 2.4 GHz CPU. We use autograd to compute the gradients.

\paragraph{Initialization.} All of our experiments were started from an initial weight vector $\w_0$ picked at random.

\paragraph{Choice of parameters for Stochastic Gradient Descent (SGD)}
We select the best mini-batch size in the set $\{\lceil 1\% \rceil, \lceil 5\% \rceil, \lceil 10\% \rceil \}$. We use a constant step-size as this yields faster initial convergence~\cite{hofmann2015variance},\cite{roux2012stochastic}. We pick the best step size in the set $\{ 1\mathrm{e}{-4}, \dots 1\mathrm{e}{-1}, 1 \}$. 
 
\paragraph{Choice of parameters for SCR and~\methodname.}
 The regularization parameter update is set with $\gamma_1=0.8, \gamma_2=1.2, \gamma_3=2$. The goal is to reduce the penalty rapidly as soon as convergence sets in, while keeping some regularization in the non asymptotic regime. A more sophisticated approach can be found in \cite{gould2012updating}. In our experiments we start with $\sigma_0=1, \eta_1 = 0.2, \text{ and } \eta_2=0.8$ as well as an initial sample size of $5\%$. We do not adapt the sample size for simplicity, although developing an adaptive scheme to select the sample size would be of practical interest.

\paragraph{Subsolver} We simply use gradient descent as a subsolver (both for SCR and~\methodname) and we pick the best step size in the set $\{ 1\mathrm{e}{-4}, \dots 1\mathrm{e}{-1}, 1 \}$. As mentioned previously, one could rely on more sophisticated solvers, e.g.~\cite{carmon2017convex,carmon2018accelerated}, but we found gradient descent to perform sufficiently well for our purpose.

\paragraph{Results}

We ran~\methodname~on three datasets whose details (number of datapoints $n$ and dimension $d$) are given in the caption of Fig.~\ref{fig:results}. All datasets are publicly available at~\url{https://www.csie.ntu.edu.tw/~cjlin/libsvmtools/datasets/}. We plot the objective value against the number of oracle evaluations in Fig~\ref{fig:exp_results}. We include a comparison to SGD and SCR~\citep{kohler2017sub}. The batch sizes and learning rates are tuned separately for each method in order to ensure a fair comparison, see description above. From these experiments, we see that STM provides significant speed-ups in terms of oracle calls (i.e. number of derivatives, including second- and third-order derivatives) over SGD and SCR. This confirms the theoretical results developed in Section~\ref{sec:analysis}.

\paragraph{Limitations}

Of course, the use of second and third-order derivatives makes the wall-clock time of second- and high-order methods more expensive and further work is necessary to make these methods truly competitive. In the second-order literature, we have recently seen a plethora of work that address this issue, including for instance KFAC~\citep{martens2015optimizing} or other efficient block-diagonal approximations~\citep{botev2017practical}. We envision that similar approximations could be use for third- and higher-order derivatives.



\vspace{-2mm}
\section{Conclusion}
\label{sec:conclusion}
We presented a sub-sampled third-order regularized optimization algorithm that finds an approximate third-order critical point and provides computational gains over the deterministic method by sub-sampling the derivatives.
We see our work as a more theoretical contribution at this stage that demonstrates that sub-sampled high-order methods can theoretically achieve faster rates of convergence. Our work also opens the door to numerous extensions that could yield to practical algorithms for optimizing complex non-convex functions.

For instance, prior work such as~\cite{allen2018natasha,xu2018first} has relied on using variance reduction to achieve faster rates. A variance-reduced variant of cubic regularization has also been shown in~\cite{wang2018stochastic} to reduce the per-iteration sample complexity and one would therefore except similar improvements can be made to the quartic model. One could also modify Algorithm~\ref{alg:stm} to rely on approximate function evaluations (instead of exact evaluations), as done in~\citep{blanchet2016convergence, bellavia2018adaptive}. Yet another extension would be to incorporate acceleration in Algorithm~\ref{alg:stm} as in~\cite{nesterov2008accelerating}.


Finally, one relevant application for the type of high-order algorithms we developed is training deep neural networks as in~\cite{tripuraneni2018stochastic, adolphs2019ellipsoidal}. An interesting direction for future research would therefore be to design a practical implementation of~\methodname~for training neural networks based on efficient tensor-vector products similarly to the fast Hessian-vector products proposed in~\cite{pearlmutter1994fast}.

\paragraph{Acknowledgements} The authors would like to thank Coralia Cartis for helpful discussions on an early draft of this paper, as well as for pointing out additional relevant work. We also thank Roman Vershynin for a discussion related to tensor concentration inequalities, as well as the reviewers whose feedback was helpful to improve this manuscript.

\bibliographystyle{plainnat}
\bibliography{main}


\newpage
\appendix
\part*{Appendix}

\section{Sampling conditions - sampling with replacement}
\label{sec:sampling_conditions}

In the main paper, we discuss two different sampling scenarios: i) random sampling \emph{without} replacement, and ii) sampling with replacement. Since sampling \emph{without} replacement yields a lower sample complexity, we directly use the corresponding results in the main part of the paper. In this appendix, we give the proofs for the case of sampling with replacement. Specifically, we prove that one can use random sampling with replacement in order to satisfy the three sampling conditions presented in Eqs.~\eqref{eq:sampling_g},~\eqref{eq:sampling_b} and~\eqref{eq:sampling_t}.

First, we introduce some known results and then derive a concentration bound for a sum of i.i.d. tensors usually a similar proof technique as in the previous subsection.

\subsection{Existing results}

The following results are well-known and can for instance be found in~\cite{tropp2012user, tropp2015introduction}.

\begin{theorem}[Matrix Hoeffding]
\label{thm:intro-hoeffding}
Consider a finite sequence $\{ \Xm_k \}$ of independent, random, self-adjoint matrices with dimension $d$, and let $\{ \Am_k \}$ be a sequence of fixed self-adjoint matrices.  Assume that each random matrix satisfies
$$
\E \Xm_k = \textbf{0} \quad\text{and}\quad \Xm_k^2 \preccurlyeq \Am_k^2
\quad\text{almost surely}.
$$
Then, for all $t \geq 0$,
$$
P \left( \lambda_{\max}\left( \sum\nolimits_k \Xm_k \right) \geq t \leq d \cdot e^{-t^2 / 8\sigma^2} \right)
	\quad\text{where}\quad
	\sigma^2 := \| \sum\nolimits_k \Am_k^2 \|.
$$
\end{theorem}


\begin{lemma}[Hoeffding's lemma]
\label{lemma:Hoeffding}
Let $Z$ be any real-valued bounded random variable such that $a \leq Z \leq b$. Then, for all $s \in \R$,
\begin{equation}
\E \left[ e^{s (Z - \E(Z))} \right] \leq \exp \left( \frac{s^2 (b - a)^2}{8} \right).
\end{equation}
\end{lemma}

\subsection{Concentration bound for sum of i.i.d. tensors}

In the following, we provide a concentration bound for sampling with replacement for tensors. This result is based on the proof technique introduced in~\cite{tomioka2014spectral} which we adapt for sums of independent random variables.

\paragraph{Proof idea} In the following, we first provide a concentration bound for each entry in the tensor $\X$ (Lemma~\ref{lem:entry_concentration}). We then use Lemma~\ref{lem:entry_concentration} to obtain a concentration bound for the tensor $\X$ by using a covering argument similar to~\cite{tomioka2014spectral}.\\




\begin{lemma}
\label{lem:entry_concentration}
Let $\X$ be a sum of $n$ i.i.d. tensors $\Y_i \in \R^{d_1 \times \dots \times d_k}$. Let $\u_1, \dots \u_k$ be such that $\| \u_i \| = 1$ and assume that for each tensor $i$, $a \leq \Y_i(\u_1,\ldots,\u_k) \leq b$. Let $\sigma := (b-a)$, then we have
\begin{align*}
P\left(|\X(\u_1,\ldots,\u_k) - \E[\X(\u_1,\ldots,\u_k)] | \geq t \right) \leq 2 \exp \left(-\frac{2t^2}{n \sigma^2}\right).
\end{align*}
\end{lemma}
\begin{proof}

By Markov's inequality and Hoeffding's lemma, we have
\begin{align*}
P&\left(\X(\u_1,\ldots,\u_k) - \E[\X(\u_1,\ldots,\u_k)] \geq t\right) \\
&= P\left(e^{s(\X(\u_1,\ldots,\u_k)- \E[\X(\u_1,\ldots,\u_k)])}\geq e^{st} \right) \\
&\leq e^{-st} \E\left[e^{(s(\X(\u_1,\ldots,\u_k)- \E[\X(\u_1,\ldots,\u_k)]}) \right] \\
&= e^{-st} \E\left[e^{(s(\sum_i \Y_i(\u_1,\ldots,\u_k)- \E[\sum_i \Y_i(\u_1,\ldots,\u_k)]}) \right] \\
&\stackrel{(i)}{=} e^{-st} \prod_{i=1}^n \E\left[e^{(s(\Y_i(\u_1,\ldots,\u_k)- \E[\Y_i(\u_1,\ldots,\u_k)]}) \right] \\
&\stackrel{(ii)}{\leq} \exp\left(-st+\frac{n \sigma^2 s^2}{8}\right),
\end{align*}
where $(i)$ follows by independence of the $\Y_i$'s and $(ii)$ follows from Lemma~\ref{lemma:Hoeffding}.

After minimizing over $s$, we obtain
\begin{equation*}
P\left(\X(\u_1,\ldots,\u_k) - \E[\X(\u_1,\ldots,\u_k)] \geq t\right) \leq e^{-2t^2/(n \sigma^2)}.
\end{equation*}

Similarly one can show that $P(\X(\u_1,\ldots,\u_k) \leq -t)\leq e^{-2t^2/(n \sigma^2)}$. We then complete the proof by taking the union of both cases.
\end{proof}

\begin{theorem}[Lemma 5.4 restated, Tensor Hoeffding Inequality]
\label{th:tensor_hoefdding_app}
Let $\X$ be a sum of $n$ i.i.d. tensors $\Y_i \in \R^{d_1 \times \dots \times d_k}$. Let $\u_1, \dots \u_k$ be such that $\| \u_i \| = 1$ and assume that for each tensor $i$, $a \leq \Y_i(\u_1,\ldots,\u_k) \leq b$. Let $\sigma := (b-a)$, then we have
\begin{equation*}
P(\norm{\X - \E \X} \geq t) \leq k_0^{(\sum_{i=1}^{k} d_i)} \cdot 2 \exp\left(-\frac{t^2}{2n \sigma^2}\right),
\end{equation*}
where $k_0 = \left(\frac{2k}{\log(3/2)}\right)$.
\end{theorem}
\begin{proof}
We use the same covering number argument as in~\cite{tomioka2014spectral}. Let $C_1,\ldots,C_k$ be $\epsilon$-covers of $S^{d_1-1}, \ldots, S^{d_k-1}$. Then since $S^{d_1-1} \times \cdots \times S^{d_k-1}$ is compact, there exists a maximizer $(\u_1^\ast,\ldots,\u_k^\ast)$ of \eqref{eq:spectralnorm}. Using the $\epsilon$-covers, we have
\begin{equation*}
\norm{\X} = \X(\bar{\u}_1+\vdelta_1,\ldots,\bar{\u}_k+\vdelta_k),
\end{equation*}
where $\bar{\u}_i \in C_i$ and $\|\vdelta_i\| \leq \epsilon$ for $i=1,\ldots,k$.
 
Now
\begin{equation*}
\norm{\X} \leq \X(\bar{\u}_1,\ldots,\bar{\u}_k) + \left(\epsilon k + \epsilon^2\binom{k}{2}+\cdots \epsilon^k\binom{k}{k}\right) \norm{\X}.
\end{equation*}

Take $\epsilon = \frac{\log(3/2)}{k}$ then the sum inside the parenthesis
 can be bounded as follows:
\begin{equation*}
\epsilon k + \epsilon^2\binom{k}{2} + \cdots \epsilon^k\binom{k}{k} \leq
\epsilon k + \frac{(\epsilon k)^2}{2!} + \cdots \frac{(\epsilon k)^k}{k!} \leq e^{\epsilon k} - 1 = \frac12.
\end{equation*}

Thus we have 
\begin{equation*}
\norm{\X} \leq 2 \max_{\bar{\u}_1\in C_1,\ldots,\bar{\u}_k\in C_k} \X(\bar{\u}_1,\ldots,\bar{\u}_k).
\end{equation*}

So far, all the steps have been identical to the proof in~\cite{tomioka2014spectral}. We conclude the proof with one last step that is a simple adaptation of the proof in~\cite{tomioka2014spectral}, combined with the result of Lemma~\ref{lem:entry_concentration}.

Since the $\epsilon$-covering number $|C_k|$ can be bounded by $\epsilon/2$-packing number, which can be bounded by $(2/\epsilon)^{d_k}$, using the union bound. Therefore, by Lemma \ref{lem:entry_concentration}
\begin{align*}
P(\norm{\X - \E \X} \geq t)&\leq \sum_{\bar{\u}_1 \in C_1, \ldots, \bar{\u}_k \in C_k} P\left(\X(\bar{\u}_1,\ldots,\bar{\u}_k) - \E[\X(\bar{\u}_1,\ldots,\bar{\u}_k]) \geq \frac{t}{2}\right) \\
&\leq k_0^{\sum_{i=1}^{k} d_i} \cdot 2 \exp\left(-\frac{t^2}{2n \sigma^2}\right).
\end{align*}

\end{proof}

\section{Auxiliary Lemmas}


\subsection{Bound regularization parameter}

The following lemma provides an upper bound on the regularization parameter $\sigma_k$. The proof is conceptually similar to Lemma 3.3 in~\cite{cartis2011adaptive2}.

\begin{lemma}
Let Assumption~\ref{a:continuity} hold and assume Condition~\ref{cond:approximate_min} holds. Also assume that
\begin{equation}
\sigma_k > \hat{\sigma}_{sup}:= \max \left(\frac{4\xi}{(1-\eta_2)} ,\frac{\xi \left(4L_t + 2 + 8 \theta \right)}{\left(1-\eta_2\right)\epsilon - 8\xi}
\right),
\label{eq:sigma_k_successful}
\end{equation}
where $\xi:=\left(\epsilon\kappa_g+\frac{\epsilon^{2/3}\kappa_b}{2}+\frac{\epsilon^{1/3}\kappa_t+L_t}{6}\right)$. Then iteration $k$ is very successful and consequently $\sigma_k \leq \gamma_3 \hat{\sigma}_{sup} := \sigma_{max}$ for all $k$.

\label{lemma:successful_sigma}
\end{lemma}
\begin{proof}

First, note that
\begin{equation}
\rho_k > \eta_2 \Longleftrightarrow r_k := f(\x_k + \s_k) - f(\x_k) - \eta_2 (\phi_k(\s_k) - f(\x_k)) < 0.
\end{equation}

We rewrite $r_k$ as
\begin{equation}
r_k = f(\x_k + \s_k) - \phi_k(\s_k) + (1 - \eta_2) (\phi_k(\s_k) - f(\x_k)).
\end{equation}

From the mean value theorem,
\begin{equation}
f(\x_k + \s_k) = f(\x_k) + \nabla f(\x_k)^{\top} \s_k+ \frac{1}{2} \s_k^\intercal  \nabla^2 f(\x_k) \s_k + \frac{1}{6} \nabla^3 f(\x_k+\alpha\s_k)[\s_k]^3
\end{equation}

for some $\alpha \in (0,1)$. Therefore we can bound the first term in $r_k$ as

\begin{equation}\label{eq:kappa_max_part1}
    \begin{aligned}
        f(\x_k + \s_k) - \phi_k(\s_k) &= \left(\nabla f(\x_k) - \g_k\right)^\intercal \s_k+\frac{1}{2} \s_k^\intercal \left(\nabla^2 f(\x_k)-B_k\right) \s_k \\
        & \qquad +\frac{1}{6} \left(\nabla^3 f(\x_k+\alpha\s_k)-\Tm_k\right)[\s_k]^3 \\
        =&\left(\nabla f(\x_k) - \g_k\right)^\intercal \s_k+\frac{1}{2} \s_k^\intercal \left(\nabla^2 f(\x_k)-B_k\right) \s_k \\& +\frac{1}{6}\left(\nabla^3 f(\x_k)-\Tm_k\right)[\s_k]^3+\frac{1}{6}\left(\nabla^3 f(\x_k+\alpha\s_k)-\nabla^3 f(\x_k)\right)[\s_k]^3\\
        \leq& \kappa_g\epsilon\|\s_k\|+\frac{1}{2}\kappa_b\epsilon^{2/3}\|\s_k\|^2+\frac{1}{6}\left(\kappa_t\epsilon^{1/3}\right)\|\s_k\|^3 +\frac{L_t}{6}\|\s_k\|^4\\
        \leq& \underbrace{\left(\epsilon\kappa_g+\frac{\epsilon^{2/3}\kappa_b}{2}+\frac{\epsilon^{1/3}\kappa_t+L_t}{6}\right)}_{:=\xi}\max(\| \s_k \|, \| \s_k \|^4)
    \end{aligned}
\end{equation}


Given that Condition \ref{cond:approximate_min} holds per assumption, we can combine Eq.~\eqref{eq:decrease_model} from Lemma~\ref{lemma:model_decrease} with the above Eq.~\eqref{eq:kappa_max_part1} to derive the upper bound $\sigma_{\sup}$ as follows.\\

\textit{Case I:} If $\|\s_k\|\geq 1$ we have
\begin{equation}
r_k < 0 \iff \sigma_k>\frac{4\xi}{(1-\eta_2)}.
\end{equation}

\textit{Case II:} If $\|\s_k\|< 1$ we have 
\begin{equation}
r_k < 0 \iff \sigma_k>\frac{4\xi}{(1-\eta_2)\|s_k\|^3}.
\label{eq:r_k_case2}
\end{equation}

We need to further simplify the RHS in the equation above that contains the term $\|s_k\|^3$. To do so, we first use Lemma~\ref{lemma:bound_s_norm} to upper bound the right hand side as
\begin{equation}
    \frac{4\xi}{(1-\eta_2)\|s_k\|^3}\leq \frac{4 \xi \kappa_k}{(1-\eta_2)(\|\nabla f(\x_k+\s_k)\| - \frac12 \epsilon)} \leq \frac{2 \cdot 4\xi \left( \sigma_k + \frac{L_t}{2} + \theta + \frac14 \right)}{(1-\eta_2)\epsilon},
\label{eq:case2_upper_bound_sigma_k}
\end{equation}
where the last inequality uses the fact that $\| \nabla f(\x_k + \s_k) \| \geq \epsilon$ before termination.\\

If $\sigma_k$ is greater than the upper bound in Eq.~\eqref{eq:case2_upper_bound_sigma_k}, it is also greater than the RHS in Eq.~\eqref{eq:r_k_case2}.

Consequently $\rho_k > \eta_2$ as soon as
\begin{equation}
\sigma_k > \frac{\xi \left(4L_t + 2 + 8 \theta \right)}{\left(1-\eta_2\right)\epsilon - 8\xi}
\end{equation}

As a result, we conclude that if $\sigma_0 < \hat{\sigma}_{sup}$, then $\sigma_k \leq \gamma_3 \hat{\sigma}_{sup}$ for all $k$, where

\begin{equation}
\hat{\sigma}_{sup}:= \max \left(\frac{4\xi}{(1-\eta_2)} ,\frac{\xi \left(4L_t + 2 + 8 \theta \right)}{\left(1-\eta_2\right)\epsilon - 8\xi}
\right)
\end{equation}

\end{proof}


\end{document}